\documentclass[11pt]{amsart}

\usepackage[margin=.8in]{geometry}

\usepackage[parfill]{parskip}
\usepackage{latexsym, amssymb, bbm}
\usepackage{amsthm}
\usepackage{graphicx}
\usepackage{amsmath}
\usepackage{lmodern}
\usepackage{psfrag}

\newtheorem{lemma}{Lemma}[section]
\newtheorem{theorem}[lemma]{Theorem}
\newtheorem{proposition}[lemma]{Proposition}
\newtheorem{corollary}[lemma]{Corollary}

\theoremstyle{definition}

\newtheorem{example}[lemma]{Example}

\DeclareMathOperator{\supp}{supp}

\DeclareMathOperator{\vo}{vol}

\DeclareMathOperator{\sgrad}{sgrad}

\DeclareMathOperator{\Cal}{Cal}

\DeclareMathOperator{\diag}{diag}

\newcommand{\ip}[1]{\langle #1 \rangle}
\newcommand{\wt}[1]{\widetilde{#1}}

\newcommand{\abs}[1]{\left| #1 \right|}
\newcommand{\norm}[1]{\left\| #1 \right\|}
\newcommand{\ind}{\boldsymbol{\mathbbm{1}}}

\def\C{\mathbb{C}}

\def\R{\mathbb{R}}
\def\Z{\mathbb{Z}}

\def\Si{\Sigma}

\def\a{\alpha}

\def\d{\partial}

\def\e{\epsilon}

\def\l{\lambda}

\def\r{\rho}
\def\s{\sigma}

\def\th{\theta}
\def\Th{\Theta}
\def\T{\mathbb{T}}

\def\vp{\varphi}

\def\w{\omega}
\def\z{\zeta}
\def\vp{\varphi}

\def\C{\mathbb{C}}
\def\CP{\mathbb{CP}}
\def\bP{\mathbb{P}}

\def\K{\mathbb{K}}

\def\La{\Lambda}

\def\H{\mathcal{H}}
\def\P{\mathcal{P}}
\def\U{\mathcal{U}}
\def\PH{\P Ham}
\def\CH{\widetilde{Ham}}

\def\cO{\mathcal{O}}

\begin{document}
\title{Quasi-states, quasi-morphisms, and the moment map}
\author{Matthew Strom Borman}
\address{Department of Mathematics\\ 
University of Chicago\\ 
Chicago, Illinois 60637}
\email{borman@math.uchicago.edu}
\thanks{This work is partially supported by the NSF-grant DMS 1006610}

\begin{abstract}
We prove that symplectic quasi-states and quasi-morphisms on a symplectic manifold descend under symplectic reduction on a superheavy level set of a Hamiltonian torus action.  Using a construction due to Abreu and Macarini, in each dimension at least four we produce a closed symplectic toric manifold with infinite dimensional spaces of symplectic quasi-states and quasi-morphisms, and a one-parameter family of non-displaceable Lagrangian tori.  By using McDuff's method of probes, we also show how Ostrover and Tyomkin's method for finding distinct spectral quasi-states in symplectic toric Fano manifolds can also be used to find different superheavy toric fibers.
\end{abstract}  

\maketitle


\section{Introduction and Results}

\subsection{An overview and statement of results}
In the series of papers \cite{EntPol03, EntPol06, EntPol08, EntPol09RS}, Entov and Polterovich introduced a way to construct symplectic quasi-states and quasi-morphisms on a closed symplectic manifold 
$(M, \w)$.  Their construction and its generalization by Usher \cite{Ush11} and Fukaya--Oh--Ohta--Ono \cite{FukOhOht11a} is based on spectral invariants in Hamiltonian Floer theory and requires the algebraic condition that some flavor of the quantum homology algebra $QH(M, \w)$ contains a field summand.  Since quantum homology is not functorial, in general there is no algebraic way to create new quasi-states and quasi-morphisms from known examples.
In \cite{Bor10} a `geometric functoriality' for quasi-states and quasi-morphisms was found, which makes no reference to quantum homology and for example lets one symplectially reduce a quasi-state on $M$ to a subcritical symplectic hyperplane section $\Si$.  In this paper we will adapt this procedure to symplectic reduction for Hamiltonian torus actions.  

\emph{Symplectic quasi-states} are functionals 
$\z: C^\infty(M) \to \R$ that satisfy the following three axioms. 
For $H,K \in C^\infty(M)$ and $a \in \R$:
	\begin{itemize}
		\item[(1)] \emph{Normalization}: $\z(1) = 1$.
		\item[(2)] \emph{Monotonicity}: If $H \leq K$, then $\z(H) \leq \z(K)$.
		\item[(3)] \emph{Quasi-linearity}: If $\{H, K\} = 0$, then $\z(H + aK) = \z(H) + a\,\z(K)$.
	\end{itemize}
Symplectic quasi-states are Lipschitz in the $C^{0}$-norm $\abs{\z(H) - \z(K)} \leq \norm{H-K}$, and
this follows from the above properties.  Symplectic quasi-states built in \cite{EntPol06, EntPol09RS, FukOhOht11a, Ush11} using spectral invariants from Hamiltonian Floer theory \cite{FukOhOht11a, Oh05, Sch00, Ush11}, also have the additional properties
\begin{itemize}
	\item[(1)] \emph{$Ham(M,\w)$-invariance}: $\z(H) = \z(H \circ \vp)$ for $\vp \in Ham(M,\w)$.
	\item[(2)] \emph{Vanishing}: $\z(H) = 0$ if $\supp(H)$ is stably displaceable.
	\item[(3)] \emph{PB-inequality}: There is a $C > 0$
	\[
		\abs{\z(H+K)-\z(H)-\z(K)} \leq C\sqrt{\norm{\{H,K\}}}
	\]
	where $\norm{\cdot}$ is the uniform norm \cite[Theorem 1.4]{EntPolZap07}
	and `PB' stands for Poisson brackets.
\end{itemize}
A subset $X \subset M$ is \emph{displaceable} if there is a $\vp \in Ham(M,\w)$ so that
$\vp(X) \cap \overline{X} = \emptyset$ and $X$ is \emph{stably displaceable} if
$X \times S^1 \subset M \times T^*S^1$ is displaceable.

One application of $Ham(M, \w)$-invariant symplectic quasi-states is to the study of displaceability of subsets \cite{BirEntPol04, EntPol06, EntPol09RS,FukOhOht11a}.  
A closed subset $X \subset M$ is \emph{superheavy} with respect to a symplectic quasi-state $\z$ if 
for all $H \in C^{\infty}(M)$
\begin{equation}\label{e: sh}
 \min_{X} H \leq \z(H) \leq \max_{X}H.
\end{equation}
So in particular if $X$ is superheavy for $\z$ and $H|_{X} = c$, then $\z(H) = c$.  Two superheavy sets of the same quasi-state $\z$ must intersect, and hence a superheavy set $X$
is non-displaceable if $\z$ is $Ham(M, \w)$-invariant.

A \emph{homogeneous quasi-morphism} on a group $G$ is a function $\mu: G \to \R$ so that 
$n\mu(g) = \mu(g^n)$ for all $n \in \Z$ and $g \in G$, and for some $D \geq 0$:
\begin{equation}\label{quasi}
	\abs{\mu(g_1 g_2) - \mu(g_1) - \mu(g_2)} \leq D \quad \mbox{for all $g_1, g_2 \in G$}.
\end{equation}
See \cite{Cal09, Kot04} for more information about quasi-morphisms.

A general construction of homogeneous quasi-morphisms on the universal cover of the group of Hamiltonian diffeomorphisms 
\[
	\mu: \CH(M, \w) \to \R
\] 
was developed in \cite{EntPol03, EntPol08, FukOhOht11a, Ost06, Ush11} also using spectral invariants.  Every element
in $\phi \in \CH(M, \w)$ can be generated by some Hamiltonian $F: M \times [0,1] \to \R$ that is normalized
in the sense that $\int_{M} F(\cdot, t)\, \w^{n} = 0$ for all $t$.  If $\phi_{F}$ denotes the the Hamiltonian isotopy generated by such an $F$, quasi-morphisms built with spectral invariants have the two additional properties.
\begin{itemize}
	\item[(1)] \emph{Stability:} For normalized $F, G: M \times [0,1] \to \R$
	\[
		\int_0^1 \min_M(F_t - G_t)\,dt \leq \frac{\mu(\phi_G) - \mu(\phi_F)}{\vo(M,\w)} \leq
		\int_0^1 \max_M(F_t-G_t)\,dt
	\]
	see \cite[Section 4.2]{EntPolZap07}.  A quasi-morphism with this property will be called
	\emph{stable}.	
	\item[(2)] \emph{Calabi Property:} If $U \subset M$ is open and stably displaceable and
	if $F: M \times [0,1] \to \R$ has support in $U \times [0,1]$, then
	\[\mu(\phi_{F}) = \Cal_U(\phi_F) := \int_0^1\int_U F_t\, \w^n dt
	\]
	where $\Cal_{U}$ is the Calabi homomorphism.
	See \cite[Theorem 1.3]{EntPol03} and \cite[Theorem 1]{Bor10}.
\end{itemize}
Due to the Calabi property, these quasi-morphisms are often referred to as Calabi quasi-morphisms.
A stable homogeneous quasi-morphisms $\mu$ on $\CH(M, \w)$ induces a symplectic quasi-state
$\z_\mu$ via
\begin{equation}\label{qs from qm}
	\z_\mu(H) = \frac{\int_M H \w^n - \mu(\phi_{H_{n}})}{\vo(M, \w)},
\end{equation}
where $H_{n} = H - \tfrac{\int_M H \w^n}{\vo(M,\w)}$ is normalized.

\subsubsection{Reduction of symplectic quasi-states and quasi-morphisms}

We can now formulate our main theorem.  
Let $(W^{2n}, \w)$ be a closed symplectic manifold equipped with a smooth map $\Phi = (\Phi_{1}, \ldots, \Phi_{k}): W \to \R^{k}$ and a regular level set $Z = \Phi^{-1}(0)$.
Suppose that the component functions $\Phi_{i}$ pairwise Poisson commute at each point in $Z$ and
that $\Phi$ induces a free Hamiltonian $\T^{k}$-action on $Z$.  Let $(M = Z/\T^{k}, \bar{\w})$ be the result of performing symplectic reduction and let $\r: Z \to M$ be the quotient map.

\begin{theorem}\label{t: main}	
	If $\z: C^{\infty}(W,\w) \to \R$ is a symplectic quasi-state with the PB-inequality
	and $Z$ is superheavy for $\z$, then $\z$ naturally induces a symplectic quasi-state 
	$$\bar{\z}:C^{\infty}(M, \bar{\w}) \to \R$$ with the PB-inequality.
	The $Ham$-invariance and vanishing properties descend from $\z$ to $\bar{\z}$.  
	If $Y \subset Z$ is superheavy for $\z$, then 
	$\r(Y) \subset M$ is superheavy for $\bar{\z}$.  
	
	Suppose $\mu: \CH(W,\w) \to \R$ is a stable homogeneous quasi-morphism and
	$Z$ is superheavy for the symplectic quasi-state $\z_{\mu}$ determined by $\mu$.  
	Then $\mu$ naturally
	induces a stable homogeneous quasi-morphism $$\bar{\mu}: \CH(M, \bar{\w}) \to \R$$  If
	$\mu$ has the Calabi property, then so does $\bar{\mu}$.
\end{theorem}
\noindent
See \eqref{e: barz} and \eqref{e: barmu} for the definitions of $\bar{\z}$ and $\bar{\mu}$,
and see Section~\ref{s: proof of mt} for the proof of Theorem~\ref{t: main}.

\subsubsection{Closed symplectic manifolds with infinite dimensional spaces of quasi-morphisms and quasi-states}

In \cite{AbrMac11}, Abreu and Macarini built many examples of non-displaceable Lagrangian tori 
$\bar{L} \subset (M, \bar{\w})$ by showing that $(M, \bar{L})$ is the result of doing symplectic reduction on 
$(W, \w)$ at a level containing a non-displaceable Lagrangian torus $L \subset W$.
So if $L$ is superheavy with respect to a symplectic quasi-state 
$\z$ on $W$ satisfying the PB-inequality, then Theorem~\ref{t: main} provides a symplectic quasi-state 
$\bar{\z}$ on $M$, for which $\bar{L}$ is superheavy.  Our second theorem will be an example application of 
Theorem~\ref{t: main} to a generalization of one the Abreu--Macarini constructions, and in 
Examples~\ref{ex:1} and \ref{ex:2} we explain how this application carries over to all of their 
examples in \cite[Sections 5 and 6]{AbrMac11}.

For positive $\a < \tfrac{1}{n+1}$, consider the $2n$-dimensional symplectic toric manifold $(Y^{2n}, \w_{\a})$ with moment polytope 
\begin{equation}\label{e: mpY}
	\Delta_{\a}^{n} = \left\{(x_{1}, \ldots, x_{n}) \in \R^{n} \mid x_{j} \geq 0\,,
	\,\, -\sum_{j=1}^{n} x_{j} + 1 \geq 0\,,\,\, 
	\sum_{j=1}^{n} x_{j} - (n-1)\a \geq 0\,,\,\, -\sum_{j=2}^{n}x_{j} + n\a \geq 0\right\}.
\end{equation}
$(Y^{2n}, \w_{\a})$ is obtained from a standard $(\CP^{n}, \w)$ with moment polytope given by
\[
	\left\{(x_{1}, \dots, x_{n}) \in \R^{n} \mid
	x_{j} \geq 0\mbox{ for all $j$},\,\, -\sum_{j=1}^{n} x_{j} + 1 \geq 0\right\}\,,
\]
by performing a small blowup at the point $(x_{1}, \ldots, x_{n}) = 0$ and a large blowup at the 
codimension two face given by 
\[
	x_{1}=0 \quad\mbox{and}\quad x_{2}+\dots+x_{n} = 1.
\]

For positive $\l < \tfrac{1-(n+1)\a}{2}$
\begin{equation}\label{e:TF}
\mbox{the fiber $L^{n}_{\l}$ over the point $(x_{1}, x_{2}, \ldots, x_{n}) = (\a+\l, \a, \ldots, \a)$}
\end{equation} 
in the moment polytope $\Delta_{\a}^n$ for $(Y^{2n}, \w_{\a})$ is non-displaceable.  
Since being non-displaceable is a closed property, it follows that the fibers over $\diag(\a)$ and
$(\tfrac{1-(n-1)\a}{2}, \a,\ldots, \a)$ are non-displaceable as well.
\begin{figure}[h]
\psfrag{1}{$x_{2}$}
\psfrag{2}{$2\alpha$}
\psfrag{3}{$\alpha$}
\psfrag{4}{$0$}
\psfrag{5}{$\tfrac{1-\alpha}{2}$}
\psfrag{6}{$1$}
\psfrag{7}{$x_{1}$}

\begin{center} 
\leavevmode 
\includegraphics[width=3in]{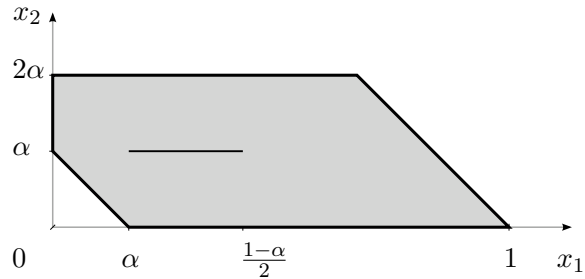}
\end{center} 

\caption{The moment polytope $\Delta_{\a}^{2}$ for $(Y^{4}, \w_{\a})$
and the interval of non-displaceable fibers, where $\a = \tfrac{1}{6}$.}
\label{f: 2blowup}
\end{figure}
This non-displaceability result was originally proven for $n=2$ by 
Fukaya--Oh--Ohta--Ono \cite[Example 10.3]{FukOhOht10}, using Lagrangian Floer theory.  In \cite[Application 7]{AbrMac11}, Abreu--Macarini show how the $n=2$ case can be proved by showing that each $L_{\l}^{n}$ is the reduction of a non-displaceable tori in a larger space.  

We will show that the Abreu--Macarini argument works for general $n$ and furthermore that each $L_{\l}^{n}$ is obtained by symplectic reduction on a level set that is superheavy for a symplectic quasi-state with the PB-inequality that comes from a stable homogeneous quasi-morphism.  This leads to the following theorem, which is proved in Section~\ref{s: AM imply T2}.

\begin{theorem}\label{t: infinite qs}
	Let $(Y^{2n}, \w_{\a}, \Delta_{\a}^{n})$ be as in \eqref{e: mpY}, for each torus 
	fiber $L_{\l}^{n}$ in \eqref{e:TF} there is a stable quasi-morphism 
	$$\mu_{\l}: \CH(Y^{2n}, \w_{\a}) \to \R$$
	for which $L_{\l}^{n}$ is superheavy with respect to the associated symplectic quasi-state 
	$\z_{\l}$.
\end{theorem}

Since the Lagrangian torus fibers $L_{\l}^{n}$ are disjoint, it follows from \eqref{e: shqm} and \eqref{e: sh}
that for any finite collection of $\l$'s the associated collections of quasi-morphisms $\mu_{\l}$ are linearly 
independent in the vector space of homogeneous quasi-morphisms and the quasi-states $\z_{\l}$ are 
linearly independent in the convex space of quasi-states.  This implies the following corollary.

\begin{corollary}
The vector space of homogeneous quasi-morphism on $\CH(Y^{2n}, \w_{\a})$ and the convex space of symplectic quasi-states on $(Y^{2n}, \w_{\a})$ are infinite dimensional.
\end{corollary}

Previous results about infinite dimensional families of 
quasi-morphisms for symplectic manifold have been limited to $B^{2n}$ and the unit ball cotangent bundle of tori $D^{*}\T^{n}$ in \cite[Theorem 1.1]{BirEntPol04}, and certain cotangent bundles
\cite[Theorem 1.3]{MonVicZap11}.  For closed manifolds the only proven examples have been that
small blowups of $\CP^{n}$ \cite[Corollary F]{OstTyo09} and $S^{2} \times S^{2}$ \cite[Theorem 1.1]{EliPol10} each having two distinct quasi-morphisms. 
Shortly after the first draft of this paper appeared, 
Fukaya--Oh--Ohta--Ono \cite[Theorem 1.10]{FukOhOht11a} produced infinite families of spectral quasi-morphisms and quasi-states for $(Y^{4}, \w_{\a})$ and other $4$-dimensional examples, and this result was announced in \cite[Remark 1.2(3)]{FukOhOht11}.  This was done by relating spectral quasi-morphisms and quasi-states from deformed Hamiltonian spectral invariants with Lagrangian Floer homology and the critical points of deformed Landau--Ginzburg potentials.

\begin{example}\label{ex:1}
The method of proof for Theorem~\ref{t: infinite qs} uses Theorem~\ref{t: main} with the following inputs:  
\begin{enumerate}
\item[(i)] the Clifford torus in $\CP^{n}$ is superheavy for a spectral quasi-state,
\item[(ii)] the fiber $L_{k, \l} \subset (X^{2n}_{k}, \w_{\l})$
from Theorem~\ref{t: cpn blow up} is superheavy for a spectral quasi-state, and
\item[(iii)] products of these quasi-states given by the comments 
in either Section~\ref{products} or Corollary~\ref{c: product}.
\end{enumerate}
Using only these inputs, the method of proof for Theorem~\ref{t: infinite qs} applies verbatim to every non-displaceable toric fiber obtained by Abreu--Macarini \cite[Section 5]{AbrMac11}.
In fact, the only fiber from (ii) that is used is 
$L_{0,\l}$ from $(X^{2n}_{0}, \w_{\l}) = (\CP^{n} \# \overline{\CP}^{n}, \w_{\l})$, a toric blow-up
of the standard $\CP^{n}$ with a small exceptional divisor.
\end{example}

\begin{example}\label{ex:2}
The non-Fano examples in \cite[Section 6]{AbrMac11} are based on the non-displaceability of the special
centered torus fiber in a weighted projective space.  Since stable quasi-morphisms or quasi-states on weighted projective spaces have not yet been constructed, the required inputs do not currently exist to directly apply Theorem~\ref{t: main} to these examples.  The use of weighted projective spaces can be avoided in the following way.  

In the first example \cite[Application 9]{AbrMac11}, Abreu--Macarini show that there is a non-displaceable toric fiber $L_{k}$ in each Hirzenbruch surface $H_{k} := \bP(\cO(-k)\oplus \C) \to \CP^{1}$ for $k \geq 2$.  
The fibers $L_{k} \subset H_{k}$ for $k \geq 2$ are stems \cite[Proposition 2.3.1]{AbrBorMcD12}, meaning that any other Lagrangian toric fiber $L \subset H_{k}$ is displaceable, and hence by Entov--Polterovich \cite[Theorem 1.8]{EntPol09RS} the fiber $L_{k}$ is superheavy for any symplectic quasi-state on $H_{k}$.  

For the other example \cite[Application 10]{AbrMac11}, one can see that it is possible to obtain the resulting manifold as the reduction of $H_{2} \times X^{4}_{0}$ such that the identified Lagrangian fiber is the reduction of $L_{k} \times L_{0,\l}$, which is superheavy for a product quasi-state using the comments in Section~\ref{products}.  Hence each Lagrangian fiber in \cite[Application 10]{AbrMac11} is superheavy for a quasi-state built by applying Theorem~\ref{t: main}.
\end{example}

It is an open question if the reduction procedure from Theorem~\ref{t: main} preserves spectral quasi-morphisms and quasi-states built by Fukaya--Oh--Ohta--Ono \cite{FukOhOht11a} and Usher \cite{Ush11}.  Namely, if a quasi-morphism $\mu$ on $(W, \w)$ is associated to the idempotent $a \in QH(W, \w)$, then is the reduction $\bar{\mu}$ on $\CH(M, \bar{\w})$ associated to some 
$\bar{a} \in QH(M, \bar{\w})$?  The Entov--Polterovich construction of quasi-morphisms 
\cite{EntPol08} using the small quantum homology algebra requires an idempotent $e \in QH_{2n}(M^{2n}, \w)$ that gives a field summand in the small quantum homology algebra over the field $\K$ of generalized Laurent series.  Since $QH_{2n}(M^{2n}, \w)$ is finite dimensional over 
$\K$, it is impossible for there to be an infinite family of Entov--Polterovich quasi-morphism for which Theorem~\ref{t: infinite qs} holds, despite the fact that the quasi-morphisms 
$\mu_{\l}$ in Theorem~\ref{t: infinite qs} are built by reducing Entov--Polterovich quasi-morphisms.
There are no such finiteness limitations when constructing quasi-morphisms using the big quantum homology algebra as in \cite{FukOhOht11a, Ush11} due to the choice of bulk-deformations.

\subsubsection{A method for finding different superheavy fibers}

Our third result, which is a necessary ingredient to the proof of Theorem~\ref{t: infinite qs}, 
demonstrates how Ostrover and Tyomkin's \cite{OstTyo09} method for finding distinct spectral quasi-states can also be used to find different superheavy toric fibers when combined with McDuff's method of probes 
\cite[Lemma 2.4]{McD09}.  The proof appears in Section~\ref{s: fsh}. 

Let $(X_{k}^{2n}, \w_{\l})$ be the toric manifold
obtained by blowing up a $k$-dimensional face in the moment polytope of $\CP^{n}$, so
$(X_{k}^{2n}, \w_{\l})$ has the moment polytope
\[
	\Delta^{n}_{k,\l} = 		
	\left\{ (x_{1}, \ldots, x_{n}) \in \R^{n} \mid x_{i} \geq 0\,,\,
		-\sum_{i=1}^{n} x_{i} + 1 \geq 0\,,\, \sum_{i=k+1}^{n} x_{i} - \l \geq 0  \right\}.
\]
\begin{theorem}\label{t: cpn blow up}
	For positive $\l < \frac{n-k-1}{n+1}\,$, the toric manifold $(X_{k}^{2n}, \w_{\l})$
	has two non-displaceable toric fibers:  The Clifford torus $L_{c}$, which is the fiber
	over $\diag(\frac{1}{n+1})$ in $\Delta^{n}_{k,\l}$, and the the fiber near the blow-up 
	$$L_{k,\l} = \{x_{1} = \dots = x_{k} = \tfrac{1}{k+1}(1-\l\tfrac{n-k}{n-k-1})\,,
	\, x_{k+1} = \cdots = x_{n} = \tfrac{\l}{n-k-1}\}.$$
	There are two symplectic quasi-states $\z_{c}$ and $\z_{k,\l}$ on $(X_{k}^{2n}, \w_{\l})$, 
	coming from stable quasi-morphisms,
	such that $L_{c}$ is superheavy for $\z_{c}$ and $L_{k,\l}$ is superheavy for $\z_{k,\l}$.
\end{theorem}
Note that when $\l = \frac{n-k-1}{n+1}$, the two fibers in Theorem~\ref{t: cpn blow up} are equal and this corresponds to the monotone case.  For large blowups $\frac{n-k-1}{n+1} \leq \l < 1$, the fiber
\begin{equation}\label{e: SiB}
L_{s} = \left\{x_{1} = \dots = x_{k} = \tfrac{1-\l}{k+2}\,,\,\,
x_{k+1} = \dots = x_{n} = \tfrac{1+(k+1)\l}{(n-k)(k+2)}\right\}
\end{equation}
is a stem, meaning that every other fiber is displaceable, which can be verified by McDuff's method of probes
\cite{McD09}.  In particular the fiber $L_{s}$ is superheavy for any symplectic quasi-state by
\cite[Theorem 1.8]{EntPol09RS}.  
For the case of blowing up a point
$(\CP^{n}\# \overline{\CP}^{n}, \w_{\l}) = (X_{0}^{2n}, \w_{\l})$,
the non-displaceability of the fiber near the blowup was proved by Cho \cite[Section 5.5]{Cho08} and
Fukaya--Oh--Ohta--Ono \cite[Example 6.2]{FukOhOht10a}, and the existence of distinct quasi-states and
quasi-morphisms was proved by Ostrover--Tyomkin \cite[Corollary F]{OstTyo09}. 

We highlight this result because its method of proof generalizes to finding superheavy fibers 
for other non-monotone symplectic toric Fano manifolds.  
For instance it is possible to show that for certain facet symmetric symplectic toric Fano manifolds considered by Maydanskiy--Mirabelli \cite{MayMir11}, there are distinct quasi-states with
disjoint superheavy Lagrangian toric fibers.
Previous explicit non-displaceability results for moment map fibers of toric manifolds that used quasi-states, tended to be in the monotone setting \cite{EntPol09RS} or required finding a stem \cite{EntPol06}.  The proof is also similar to the methods used in Lagrangian Floer homology that relate critical points of the 
Landau--Ginzburg potential, and its various deformations, to non-displaceable fibers of the moment map of a symplectic toric manifold \cite{FukOhOht10a,FukOhOht10,FukOhOht11,FukOhOht11a,WilWoo11,Woo11}.


\subsection{Notations and Conventions}

In this paper $(M^{2n}, \w)$ will always be a closed symplectic manifold.  A Hamiltonian
$H \in C^\infty(M)$ determines a vector field $\sgrad H$ on $M$ by
\[\w(\sgrad H, \cdot) = -dH\]  
and in this manner any time-dependent Hamiltonian $F: M \times [0,1] \to \R$ gives an isotopy 
$\phi_F = \{f_t\}_{t \in [0,1]}$.  The collection of all maps $f_1$ obtained this way is the \emph{Hamiltonian group} $Ham(M,\w)$.  

Denote by $\H(M,\w) \subset C^\infty(M)$ the set of functions \emph{normalized} to have mean zero 
$\int_M H \w^n = 0$ and $\H(M,\w)$ can be thought of as the Lie algebra of $Ham(M,\w)$ with the Poisson bracket
\[
	\{H, K\} = \w(\sgrad K, \sgrad H) = dH(\sgrad K).
\]
The space of smooth paths based at the identity $\PH(M,\w)$, can be identified with $\P\H(M,\w)$, the 
space of functions $F: M \times [0,1] \to \R$ such that $F_t \in \H(M,\w)$ at all times.  
The group structure of time-wise product on $\PH(M,\w)$ carries over to $\P\H(M,\w)$ as
$\phi_F\phi_G = \phi_{F\#G}$ and $\phi_F^{-1} = \phi_{\bar{F}}$
where
\[
	(F\#G)(x,t) = F(x,t) + G(f_t^{-1}(x),t) \quad \mbox{and} \quad \bar{F}(x,t) = -F(f_t(x),t).
\]
The universal cover $\CH(M,\w)$ is $\PH(M,\w)$ where paths are considered up to
homotopy with fixed endpoints.

\subsection{Symplectic quasi-states and quasi-morphisms in symplectic topology}

We will start by briefly sketching the construction for quasi-states and quasi-morphisms using spectral invariants from Hamiltonian Floer homology and the quantum homology algebra $QH(M,\w)$, as developed in \cite{EntPol03,EntPol06,Ost06, EntPol08, Ush11, FukOhOht11a}.  We will be a bit vague, since while the outline below remains the same, the conventions and types of spectral invariants vary between authors.
Given an element $a \in QH(M, \w)$ in the quantum homology algebra, there is an associated spectral invariant defined in terms of Hamiltonian Floer theory,
which is a functional \[c(a, \cdot): C^\infty(M \times [0,1]) \to \R.\]
These spectral invariants have the inequality
\[
	c(a \ast b, F \# G) \leq c(a, F) + c(b, G)
\]
where $a \ast b$ is the quantum product in $QH(M, \w)$.  Therefore if $e$ is an idempotent, 
$e = e\ast e$, then one has a triangle inequality
\[
	c(e, F \# G) \leq c(e, F) + c(e, G).
\]
For an idempotent $e$, one can form $\mu(e, \cdot): C^\infty(M \times [0,1]) \to \R$ where
\begin{equation}\label{qm def}
	\mu(e, F) = \int_0^1\int_M F(x,t)\, \w^n dt -\vo(M,\w) \lim_{k\to \infty} \frac{c(e,F^{\#k})}{k},
\end{equation}
which descends to a function
\begin{equation}
	\mu(e, \cdot): \CH(M,\w) \to \R.
\end{equation}

As it is nicely laid out in \cite[Theorem 1.4]{Ush11}, if $e \in QH(M, \w)$ is an idempotent and there is
a uniform bound for the associated spectral norm, meaning that for all $F \in C^\infty(M \times [0,1])$
\begin{equation}\label{e: bound}
	c(e, F) + c(e, \bar{F}) \leq C,
\end{equation}
then $\mu(e, \cdot): \CH(M, \w) \to \R$ is a homogeneous quasi-morphism.  As observed by McDuff and explained in \cite{EntPol08}, the arguments in \cite{EntPol03} show that if an idempotent $e$ splits off a field summand from $QH(M, \w)$ then \eqref{e: bound} is satisfied and hence
$\mu(e, \cdot)$ is a quasi-morphism.  We will call any quasi-morphism built this way a \emph{spectral quasi-morphism}.  Using \eqref{qs from qm}, such spectral quasi-morphisms induce \emph{spectral quasi-states} via
\begin{equation}\label{spec qs}
	\z(e, \cdot): C^{\infty}(M) \to \R \quad\mbox{where}\quad 
	\z(e, F) = \lim_{k\to \infty} \frac{c(e,k\,F)}{k}
\end{equation}
Usher has proved that spectral quasi-states and quasi-morphisms exist on any closed symplectic toric manifold and on any closed symplectic manifolds blown up at a point \cite[Theorem 1.6]{Ush11}.  Recently using an entirely different construction, Shelukhin built a quasi-morphism on $\CH(M, \w)$ for any closed symplectic manifold \cite[Corollary 1]{She11}.  However Shelukhin's quasi-morphisms are not stable and do not induce quasi-states, so Theorem~\ref{t: main} does not apply to them.

While applications of quasi-morphisms have tended to focus on the algebraic structure of 
$\CH(M,\w)$ and its geometry with respect to the Hofer metric \cite{BirEntPol04,EntPol03, EntPolPy09, Le-10, McD10, Pol06}, applications of symplectic quasi-states have been geared towards studying various rigidity phenomenon in symplectic topology.  For instance the PB-inequality is a manifestation of the $C^{0}$-rigidity of Poisson brackets first observed in \cite{CarVit08} and it is the main tool used to lower bound the Poisson bracket invariants recently introduced in \cite{BuhEntPol12}. 

The other application of symplectic quasi-states has been to the study of displaceability of subsets via Hamiltonian diffeomorphisms $Ham(M, \w)$, which has been undertaken in \cite{BirEntPol04, EliPol10,EntPol06, EntPol09RS,FukOhOht11a}.  As explained above \eqref{e: sh}, there is the notion of a closed subset $X \subset M$ being \emph{superheavy} with respect to a symplectic quasi-state $\z$, and
being superheavy implies non-displaceable if $\z$ is $Ham(M, \w)$-invariant.
Example results proved with this method are that the moment map for any finite dimensional Poisson commuting subspace of $C^{\infty}(M)$ must have a non-displaceable fiber \cite[Corollary 2.2]{EntPol06} and for Hamiltonian $\T^{k}$-actions on a monotone $(M^{2n}, \w)$ a special non-displaceable fiber is identified \cite[Theorem 1.11]{EntPol09RS}.

Finally we note that the inequality \eqref{e: sh} defining a set to be superheavy with respect to a symplectic quasi-state has a corresponding notion for stable quasi-morphisms on $\CH(M, \w)$.
\begin{proposition}\label{p: superheavy qm}
	Let $\mu: \CH(M) \to \R$ be a stable homogeneous quasi-morphism.
	A closed subset $X \subset M$ is superheavy with respect to the associated symplectic 
	quasi-state $\z_{\mu}$ if and only if $\mu$ restricted to 
	$\CH_{M\setminus X}(M)$ is the Calabi homomorphism. 
	In general we have the bounds
	\begin{equation}\label{e: shqm}
		-\max_{X,t} F_{t} \leq \frac{\mu(\phi_{F})}{\vo(M)} \leq -\min_{X, t} F_{t}
	\end{equation}
	for any $F \in \P\H(M)$ if $X \subset M$ is superheavy for $\z_{\mu}$.
\end{proposition}
The `if' part of the `if and only if' claim follows directly from the definition of $\z_{\mu}$.
This proposition, which is proven in Section~\ref{s: shqm}, shows that the Calabi property for $\mu$ and the vanishing property for $\z_{\mu}$ are the same thing.

\subsection*{Acknowledgments}
I am very grateful to Miguel Abreu and Leonardo Macarini for providing me with a preliminary version of their paper \cite{AbrMac11}, which along with my discussions with them was the motivation for this work.
I would like to thank my advisor Leonid Polterovich for pointing out the connection between Abreu and Macarini's work and my previous paper \cite{Bor10}, and for his wonderful guidance which significantly
improved the presentation and the content of this paper.  I would also like to thank Michael Usher, the organizer of the 2011 Georgia Topology Conference, and Yann Rollin, Vincent Colin, and Paolo Ghiggini, the organizers of the Conference on Contact and Symplectic Topology (Nantes, June 2011), for giving me the opportunity to present this work and for organizing such great conferences.  Finally I would also like to thank the anonymous referee for their comments and corrections.


\section{Proving Theorem~\ref{t: main}}\label{s: proof of mt}

In this section, let $(W^{2n}, \w)$ be a closed symplectic manifold equipped with a
smooth map $\Phi: W \to \R^{k}$, a regular level set $Z = \Phi^{-1}(0)$ such that
all component functions $\Phi_{i}$ Poisson commute on $Z$, and 
$\Phi$ induces a free Hamiltonian $\T^{k}$-action on $Z$.  Let $(M = Z/\T^{k}, \bar{\w})$ be the result of performing symplectic reduction and let $\r: Z \to M$ be the quotient map.
As we will explain in Section~\ref{s:LM}, without loss of generality we can assume that $\Phi$ induces a free 
Hamiltonian $\T^{k}$-action in a neighborhood of $Z$ without changing the original free Hamiltonian 
$\T^{k}$-action on $Z$.  It follows from the equivariant coisotropic neighborhood theorem that any two such models are locally $\T^{k}$-equivariantly symplectomorphic near $Z$.

The proof of Theorem~\ref{t: main} will be in the spirit of \cite{Bor10}, so we will introduce a linear, order preserving map in Section~\ref{s:LM}
\begin{equation}\label{e: Th}
	\Th: C^{\infty}(M) \to C^{\infty}(W)
\end{equation}
in order to pull quasi-states and quasi-morphisms for $W$ back to $M$.  
The main properties of $\Th$ are collected into the following lemma, which is proved in Section~\ref{s: lemmaproof}.

\begin{lemma}\label{l: l1}
	The map $\Th$ preserves the property of having zero mean and hence can be viewed as a map
	\begin{equation}\label{e:TH}
		\Th: \P\H(M) \to \P\H(W).
	\end{equation}
	Functions in the image of $\Th$ Poisson commute with $\Phi$
	\begin{equation}\label{e: Th and Phi}
		\{\Th(F), \Phi_{i}\} = 0 \quad\mbox{on $W$}.
	\end{equation}
	Therefore a Hamiltonian diffeomorphism generated by a Hamiltonian
	in the image of \eqref{e:TH} preserves $Z$ and all other level sets of $\Phi$.  
	
	At points in $Z$, the map $\Th$ acts like $\r^{*}: C^{\infty}(M) \to C^{\infty}(Z)$
	and respects the Poisson brackets, meaning
	\begin{equation}\label{e: Th on Z}
		\Th(F)|_{Z} = F \circ \r \quad\mbox{and}\quad \{\Th(F), \Th(G)\}|_{Z} = \Th(\{F, G\})|_{Z}. 
	\end{equation}
	If the Hamitlonian isotopies $\{g_{t}\} \in \PH(M)$ and $\{\wt{g}_{t}\} \in \PH(W)$
	are generated by $G_{t}$ and $\Th(G_{t})$, then 
	\begin{equation}\label{e: rho and g}
		g_{t} \circ \r = \r \circ \wt{g}_{t} : Z \to M.
	\end{equation}
	For $F \in C^{\infty}(M)$, if $\supp(F) \subset M$ is (stably) displaceable in $M$,
	then $\supp(\Th(F)) \subset W$ is (stably) displaceable in $W$.
	
	The term measuring the failure of $\Th: \P\H(M) \to \P\H(W)$ to be a homomorphism
	\begin{equation}\label{e: failure}
	\overline{\Th(F\#G)}\,\#\,(\Th(F)\#\Th(G)) :  W \times [0,1] \to \R	
	\end{equation}
	vanishes on $Z$.  This also holds for larger products as well,
	in particular for $\overline{\Th(F^{\#k})}\,\#\,(\Th(F)^{\#k})$.		
\end{lemma}

\subsection{Theorem~\ref{t: main} for symplectic quasi-states}

Let $\z: C^{\infty}(W) \to \R$ be a symplectic quasi-state with the PB-inequality and assume that
our regular level set $\Phi^{-1}(0) = Z$ is superheavy with respect to $\z$.  
For any $\Th$ as in \eqref{e: Th}, define the functional
\begin{equation}\label{e: barz}
	\bar{\z}: C^{\infty}(M) \to \R \quad \mbox{by}\quad \bar{\z}(F) = \z(\Th(F))
\end{equation}
to be the pullback of $\z$ by $\Th$.  We will need the following lemma, which is proved in Section~\ref{s: lemmaproof}, to prove that $\bar{\z}$ is a symplectic quasi-state.
Note that the second claim in Lemma~\ref{l: PBZ} proves that $F \mapsto \z(\Th(F))$ is independent
of $\Th$, provided that $\Th$ satisfies Lemma~\ref{l: l1}.

\begin{lemma}\label{l: PBZ}
	If $H,K \in C^{\infty}(W)$ Poisson commute with $\Phi$, then
	\begin{equation}\label{e: PBZ}
		\abs{\z(H+K) - \z(H) -\z(K)} \leq C \sqrt{\norm{\{H,K\}|_{Z}}}
	\end{equation}
	and if $H=K$ on $Z$ as well, then $\z(H) = \z(K)$.
\end{lemma}

\begin{proof}[Proof that $\bar{\z}$ is a symplectic quasi-state]
The normalization condition for $\bar{\z}$ follows from
the fact \eqref{e: Th on Z} that $\Th(F)|_{Z} = F \circ \r$ and that $Z$ is superheavy
for $\z$.  The monotonicity condition follows by construction.

For $F,G \in C^{\infty}(M)$, it follows from Lemma~\ref{l: l1} that
$\Th(F)$ and $\Th(G)$ Poisson commute with $\Phi$ and hence by Lemma~\ref{l: PBZ}
that
\begin{align*}
	\abs{\bar{\z}(F + G) -\bar{\z}(F) - \bar{\z}(G)} &=
	\abs{\z(\Th(F) + \Th(G)) - \z(\Th(F)) - \z(\Th(G))}\\ 
	&\leq C \sqrt{\norm{\{\Th(F), \Th(G)\}|_{Z}}}\\
	&= C \sqrt{\norm{\{F, G\}}}
\end{align*}
where we used \eqref{e: Th on Z} for the last line.  Therefore $\bar{\z}$ has the PB-inequality,
which implies quasi-additivity.
\end{proof}

\begin{proof}[Proof of additional properties of $\bar{\z}$]
Suppose that $X \subset Z$ is superheavy for $\z$.
For any function $F \in C^{\infty}(M)$ such that $F|_{\r(X)} \geq c$, then by \eqref{e: Th on Z} we have
that $\Th(F)|_{X} \geq c$.  Therefore since $X$ is superheavy for $\z$ it follows that
\[
	\bar\z(F) = \z(\Th(F)) \geq c
\] and hence $\r(X)$ is superheavy for $\bar{\z}$.

Let $g_{t} \in Ham(M)$ be generated by $G_{t} \in C^{\infty}(M)$ and let $\wt{g}_{t} \in Ham(W)$
be generated by $\Th(G_{t})$.  If follows from \eqref{e: Th on Z} and \eqref{e: rho and g} in
Lemma~\ref{l: l1} that on $Z$
\[
	\Th(F \circ g_{t}) = \Th(F) \circ \wt{g}_{t}.
\]
Therefore if $\z$ is $Ham(W)$-invariant, then by Lemma~\ref{l: PBZ} if follows that
\[
	\bar{\z}(F \circ g_{1}) = \z(\Th(F \circ g_{1})) = \z(\Th(F) \circ \wt{g}_{1}) = \z(\Th(F)) = \bar{\z}(F)
\]
so $\bar{\z}$ is $Ham(M)$-invariant.

That the (stable) vanishing property passes from $\z$ to $\bar{\z}$ follows from the last item in Lemma~\ref{l: l1}.  The claim about superheavy sets follows from construction due to the first item in
\eqref{e: Th on Z}.
\end{proof}

\subsection{Theorem~\ref{t: main} for stable quasi-morphisms}

Let $\mu: \CH(W) \to \R$ be a stable homogeneous quasi-morphism and assume that
our regular level set $\Phi^{-1}(0) = Z$ is superheavy with respect to the quasi-state $\z_{\mu}$
determined by $\mu$.  
For any $\Th$ as in \eqref{e: Th}, define
\begin{equation}\label{e: barmu}
	\bar{\mu}: \CH(M) \to \R \quad \mbox{by}\quad 
	\bar{\mu}(\vp) = \tfrac{\vo(M)}{\vo(W)}\, \mu(\phi_{\Th(F)})
\end{equation}
where $F \in \P\H(M)$ is any Hamiltonian generating $\vp \in \CH(M)$.  The constant
$\tfrac{\vo(M)}{\vo(W)}$ ensures that $\bar{\mu}$ will have the stability property with the constant
$\vo(M)$.  

Observe that if $\Th$ and $\Th'$ both satisfy Lemma~\ref{l: l1},
then $$\overline{\Th(F)} \# \Th'(F) \quad\mbox{vanishes on $Z$}.$$
By Proposition~\ref{p: superheavy qm} and the quasi-morphism property of $\mu$,
independently of $F$, $\mu(\Th(F))$ and $\mu(\Th'(F))$ are a bounded distance apart.
Therefore if $\bar\mu$ is a homogenous quasi-morphism, then it 
is independent of the $\Th$ used, provided $\Th$ satisfies Lemma~\ref{l: l1}.

The proof that $\bar{\mu}$ defines a stable homogeneous quasi-morphism is similar to the proof of \cite[Theorem 4]{Bor10}, where the group theory lemmas in \cite[Lemma 17 and 18]{Bor10} are combined with the following lemma, which is proved in Section~\ref{s: lemmaproof} and generalizes \cite[Lemma 21]{Bor10}.
\begin{lemma}\label{l: l2}	
	Let $(W_{1}, \w_{1})$ and $(W_{2}, \w_{2})$ be compact symplectic manifolds
	and let $Z \subset W_{2}$ be a closed submanifold.
	Suppose that $\Th: \P\H(W_{1}) \to \P\H(W_{2})$ is a linear map such that
	for any $F \in \H(W_{1})$ the vector field $\sgrad \Th(F)_{z}$ is tangent to $Z$ for all $z \in Z$,
	and for $F,G \in \H(W_{1})$
	\[
		\Th(\{F,G\})|_{Z} = \{\Th(F), \Th(G)\}|_{Z}.
	\]
	If $F \in \P\H(W_{1})$ generates a null homotopic loop $[\phi_{F}] = \ind$ in
	$\CH(W_{1})$, then 	
	\begin{equation}\label{e: homotopy}
		[\phi_{\Th(F)}] = [\phi_{K}] \quad\mbox{as elements of}\quad \CH(W_{2}) 
	\end{equation}	
	for some $K \in \P\H(W_{2})$ that vanishes on $Z$.
\end{lemma}

\begin{proof}[Proof that $\bar{\mu}$ is a stable homogeneous quasi-morphism]
	It follows from \eqref{e: failure} in Lemma~\ref{l: l1} and 
	\eqref{e: shqm} in Proposition~\ref{p: superheavy qm} that
	$$
		\mu(\phi_{\Th(F\#G)}^{-1}\phi_{\Th(F)}\phi_{\Th(G)}) = 0.
	$$
	Therefore by \cite[Lemma 17]{Bor10} the pullback of $\mu$ by
	$\Th: \PH(M) \to \PH(W)$
	$$ \Th^{*}\mu: \PH(M) \to \R \quad\mbox{by}\quad F \mapsto \mu(\phi_{\Th(F)})$$
	is a homogeneous quasi-morphism.
	
	It follows from \eqref{e: homotopy} in Lemma~\ref{l: l2} and
	\eqref{e: shqm} in Proposition~\ref{p: superheavy qm} that
	$\Th^{*}\mu: \PH(M) \to \R$ vanishes on elements in the kernel of the
	quotient map $\PH(M) \to \CH(M)$.
	Hence by \cite[Lemma 18]{Bor10}, $\Th^{*}\mu$ descends 
	to a homogenous quasi-morphism
	$$
		\bar{\mu}: \CH(M) \to \R
	$$
	that after rescaling is given by \eqref{e: barmu}.
	
	The stability of $\bar\mu$ follows from the stability of $\mu$ since for normalized
	functions $F,G \in \H(M)$
	\[
		\min_{M}(F-G) = \min_{W}(\Th(F) - \Th(G))
	\]
	and likewise for $\max$.	
\end{proof}

\begin{proof}[Proof that the Calabi property passes from $\mu$ to $\bar{\mu}$.]
	By checking on normalized Hamiltonian, one can verify that
	the quasi-state $\z_{\bar\mu}$ formed from $\bar{\mu}$ and the quasi-state 
	$\overline{\z_{\mu}}$ formed by reducing $\z_{\mu}$ are equal.
	If $\mu$ has the Calabi property, then $\z_{\mu}$ has the vanishing property
	and hence so does $\overline{\z_{\mu}} = \z_{\bar\mu}$.  Therefore
	by Proposition~\ref{p: superheavy qm} it follows that $\bar\mu$
	has the Calabi property.
\end{proof}

\subsection{Products for symplectic quasi-states and quasi-morphisms}\label{products}

Any na\"{\i}ve notion of taking two symplectic quasi-states $\z_{1}$ on $(M_{1}, \w_{1})$ and
$\z_{2}$ on $(M_{2}, \w_{2})$, and forming their \emph{product symplectic quasi-state} 
$\z_{1} \boxtimes \z_{2}$ on $(M_{1} \times M_{2}, \w_{1} \oplus \w_{2})$ would include the following property
\begin{equation}\label{e: product qs}
	(\z_{1}\boxtimes\z_{2})(F_{1} + F_{2}) = \z_{1}(F_{1}) + \z_{2}(F_{2})
\end{equation}
where $F_{i} \in C^{\infty}(M_{i})$.  As shown in the proof of \cite[Theorem 1.7]{EntPol09RS}
if $X_{i} \subset M_{i}$ is superheavy for $\z_{i}$, then 
property \eqref{e: product qs} implies $X_{1} \times X_{2} \subset M_{1} \times M_{2}$ is superheavy for 
$\z_{1} \boxtimes \z_{2}$.
The corresponding identity for quasi-morphisms $\mu_{i}$ on $\CH(M_{i}, \w_{i})$ is
\begin{equation}\label{e: product qm}
	(\mu_{1} \boxtimes \mu_{2})(\phi_{F_{1}+F_{2}}) = 
	\mu_{1}(\phi_{F_{1}}) + \mu_{2}(\phi_{F_{2}})
\end{equation}
for $F_{i} \in \P\H(M_{i}, \w_{i})$.

In general there is no way to form the products $\z_{1} \boxtimes \z_{2}$ and 
$\mu_{1} \boxtimes \mu_{2}$ for abstract symplectic quasi-states and quasi-morphisms,
but in favorable circumstances one can form the product of spectral quasi-states and quasi-morphisms.
Suppose one has that
\begin{equation}\label{e: tensor}
QH_{2n_{1}}(M_{1}, \w_{1}) \otimes_{\K} QH_{2n_{2}}(M_{2}, \w_{2}) = QH_{2n_{1}+2n_{2}}(M_{1} \times M_{2}, \w_{1} \oplus \w_{2})
\end{equation}
as $\K$-algebras, where $\K$ is algebraically closed.  Then if $e_{i} \in QH_{2n_{i}}(M_{i}, \w_{i})$
split off fields, then they must be $1$-dimensional since $\K$ is algebraically closed and \eqref{e: tensor}
ensures that $e_{1} \otimes e_{2}$ still splits off a field.  In this case, it follows from
\cite[Theorems 1.7 and 5.1]{EntPol09RS} that products such as \eqref{e: product qs} and \eqref{e: product qm} exist for spectral quasi-states and quasi-morphisms using the Entov--Polterovich construction.  In Corollary~\ref{c: product} below we give a different proof that such products always exist for spectral quasi-states and quasi-morphisms using the Entov--Polterovich construction in the case of symplectic toric Fano manifolds.

It turns out that the property of being able to form products such as \eqref{e: product qs} and 
\eqref{e: product qm} is preserved by the reduction procedure of Theorem~\ref{t: main}.
\begin{proposition}
	In the setting of Theorem~\ref{t: main}, suppose that the symplectic quasi-states
	$\bar{\z}_{i}$ on $(M_{i}, \bar{\w}_{i})$ are the reduction of symplectic quasi-states
	$\z_{i}$ on $(W_{i}, \w_{i})$. Suppose that there is a product symplectic quasi-state
	$\z_{1} \boxtimes \z_{2}$ on $(W_{1} \times W_{2}, \w_{1} \oplus \w_{2})$, which satisfies 
	\eqref{e: product qs} and the PB-inequality.  
	Then the reduction $\overline{\z_{1} \boxtimes \z_{2}}$ defines a product symplectic quasi-state
	$\bar{\z}_{1} \boxtimes \bar{\z}_{2}$ that satisfies \eqref{e: product qs}.  
	The analogous result holds for stable quasi-morphisms.
\end{proposition}
\begin{proof}
	If $Z_{i} = \Phi_{i}^{-1}(c_{i})$ are the respective superheavy regular level sets, which one reduces
	to form $(M_{i}, \bar{\w}_{i})$, then 
	$Z_{1} \times Z_{2} = (\Phi_{1} \times \Phi_{2})^{-1}(c_{1}, c_{2})$ is a regular level set and
	it is superheavy for $\z_{1} \boxtimes \z_{2}$.  Therefore Theorem~\ref{t: main} applies and
	one can form its reduction $\overline{\z_{1} \boxtimes \z_{2}}$, which will be 
	a symplectic quasi-state on $(M_{1} \times M_{2}, \w_{1}\oplus \w_{2})$.  Now for
	$F_{i} \in C^{\infty}(M_{i})$, one has that
	\begin{align*}
		(\overline{\z_{1} \boxtimes \z_{2}})(F_{1} + F_{2}) 
		&= (\z_{1} \boxtimes \z_{2})(\Th(F_{1}) + \Th(F_{2})) &\mbox{by definition}\\
		&= (\z_{1} \boxtimes \z_{2})(\Th_{1}(F_{1}) + \Th_{2}(F_{2}))\\
		&= \z_{1}(\Th_{1}(F_{1})) + \z_{2}(\Th_{2}(F_{2})) &\mbox{by \eqref{e: product qs}}\\
		&= \bar{\z}_{1}(F_{1}) + \bar{\z}_{2}(F_{2}) &\mbox{by definition}
	\end{align*}
	where in the second equality, we switch from cutoff functions centered on
	$Z_{1} \times Z_{2}$ to cutoff functions centered on $Z_{1} \times W_{2}$
	and $W_{1} \times Z_{2}$.  This is permissible since $Z_{1} \times Z_{2}$ is superheavy for
	$\z_{1} \boxtimes \z_{2}$.  The proof for stable quasi-morphisms is analogous.
\end{proof}


\subsection{A local model for regular level sets of moment maps and the construction of the map
$\Th$}\label{s:LM}

As described in \cite[Example 2.3]{Gin07}, our regular level set $Z$, on which $\Phi$ induces a free 
Hamiltonian $\T^{k}$-action, is a stable coisotropic submanifold of $W$.  
In particular there are $\T^{k}$-invariant $1$-forms $\a_{1}, \ldots, \a_{k}$ on $Z$ so that for 
$\w_{0} = \w|_{Z}$
\begin{equation}\label{e: alphas}
		\a_{i}(\sgrad \Phi_{j}) = \delta_{ij} \quad \ker \w_{0} \subset \ker d\a_{i} \quad
		\a_{1} \wedge \ldots \wedge \a_{k} \wedge (\w_{0})^{n-k} \not= 0\,.
\end{equation}
Let $\pi: Z \times \R^{k} \to Z$ and $r:Z \times \R^{k} \to \R^{k}$ be the projections, 
then in a neighborhood $\wt{\U}$ of $Z \times 0 \subset Z \times \R^{k}$ the following $2$-form
is symplectic
\begin{equation}\label{e: local}
\wt{\w} = \pi^{*}\w_{0} + \sum_{i=1}^{k} d(r_{i}\,\pi^{*}\a_{i}) \quad \mbox{with
$(\sgrad r_{i})_{(z,r)} = (\sgrad \Phi_{i})_{z}$ and $\wt{\w}|_{Z \times 0} = \w_{0}$.}
\end{equation}
Furthermore the neighborhood can be chosen so that 
$$
r: (\wt{\U}, \wt{\w}) \to \R^{k}
$$
is the moment map for a free Hamiltonian $\T^{k}$-action, since the Hamitlonian action of $r$ on the
level set $Z \times p$ is the same as $\Phi$'s Hamiltonian action on $Z$.
By the coisotropic neighborhood theorem there is a symplectomorphism
\begin{equation}\label{e: equivar}
\psi: (\U, \w) \to (\wt{\U}, \wt{\w}) \quad\mbox{that is the identity on $Z$}
\end{equation}
for some neighborhood $\U$ of $Z \subset W$ and by replacing $\Phi$ with
$$
r \circ \psi: (\U, \w) \to \R^{k},
$$ 
we can assume that $\Phi$
gives a free Hamiltonian $\T^{k}$-action in a neighborhood of $Z$, as was promised at the beginning
of Section~\ref{s: proof of mt}.

Let $(M = Z/\T^{k},\bar{\w})$ be the result of applying symplectic reduction and let
$\r: Z \to M$ be the quotient map.  Given $F \in C^{\infty}(M)$, we can lift it to the $\T^{k}$-invariant
function $\r^{*}F$ on $Z$ and then $\pi^{*}\r^{*}F$ on $\wt{\U}$.
Let $\th: \R^{k} \to [0,1]$ be a smooth function supported near $r=0$ such that $\th(0) = 1$, then
we can define
\begin{equation}\label{e: local Th}
	\Th: C^{\infty}(M) \to C^{\infty}_{c}(\wt{U}) \quad\mbox{by}\quad \Th(F) = \th \cdot \pi^{*}\r^{*}F
\end{equation}
where $\th$ is a function of the $r$ variable and $\pi^{*}\r^{*}F$ is a function on $\wt{\U}$.  Using the symplectomorphism $\psi$, we can view $\Th$ as a map
\[
	\Th: C^{\infty}(M) \to C^{\infty}(W)
\]
and this will be the $\Th$ in \eqref{e: Th}.

\subsection{Proofs of Lemmas~\ref{l: l1}, \ref{l: PBZ}, and \ref{l: l2}}\label{s: lemmaproof}
The first lemma will be proved in a local model $\wt{U} \subset Z \times \R^{k}$, where
$\Th: C^{\infty}(M) \to C^{\infty}(\wt{U})$ is given by \eqref{e: local Th}.
For ease of exposition we will think of $\wt{U}$ as $Z \times \R^{k}$ and recall that under the local
model $\Phi$ is identified with the projection $r: Z\times \R^{k} \to \R^{k}$.

\begin{proof}[Proof of Lemma~\ref{l: l1}]
	In the local model $(Z \times \R^{k}, \wt{\w})$ from \eqref{e: local},
	$$\wt{\w}^{n} = \binom{n}{k}\, dr_{1}\wedge d\a_{1}\wedge \dots \wedge dr_{k}\wedge \a_{k} \wedge
	\w_{0}^{n-k}$$
	so integration over the fiber gives
	\begin{align*}
	\int_{Z \times \R^{k}} \Th(F)\, \wt{\w}^{n} 
	&= \binom{n}{k} \int_{\R^{k}}\th(r)\,dr \int_{Z} \r^{*}F\, d\a_{1}\wedge \dots \wedge d\a_{k} 
	\wedge \w_{0}^{n-k}\\
	&= \binom{n}{k} \int_{\R^{k}}\th(r)\,dr \int_{M} F\,\bar{\w}^{n-k}
	\end{align*}
	This can be summarized as
	\begin{equation}\label{e: integral}
		\int_{W} \Th(F)\,\w^{n} = \frac{\int_{W}\Th(1)\,\w^{n}}{\vo(M)} \int_{M}F\,\bar{\w}^{n-k}
	\end{equation}
	and hence $\Th$ preserves the property of functions having zero mean.
	
	The relation \eqref{e: Th and Phi} holds because $r_{1}, \ldots, r_{k}, \r^{*}F$
	pairwise Poisson commute.  The first claim in \eqref{e: Th on Z} follows by construction,
	and the second follows since at points in $Z = Z \times 0 \subset Z \times \R^{k}$,
	\begin{equation}\label{e: sgrad on Z}
		\r_{*}\pi_{*}\sgrad(\th \cdot \pi^{*}\r^{*}F)_{(z,0)} = \sgrad(F)_{\r(z)}.
	\end{equation}
	The identify \eqref{e: rho and g} follows from \eqref{e: sgrad on Z} since 
	in the local model
	$$\d_{t}(\r \circ \pi \circ \wt{g}_{t}) = \sgrad(G_{t})_{\r\circ\pi\circ \wt{g}_{t}}$$
	at points in $Z$, which is the same ODE that $g_{t} \circ \r \circ \pi$ satisfies.
		
	It follows from \eqref{e: rho and g} that $X \subset M$ is (stably) displaceable
	only if $$\r^{-1}(X) \subset Z \times \R^{k}$$ is (stably) displaceable.  By picking a $\th$ with
	small support, we can make it so that $\supp(\Th(F))$ is contained in any neighborhood of
	$$\r^{-1}(\supp(F)) \subset Z \times \R^{k}$$ so if $\supp(F) \subset M$ is (stably) displaceable,
	then so is $\supp(\Th(F))$.
	
	For \eqref{e: failure}, let $f_{t}$, $\wt{f}_{t}$, and $\wt{(fg)}_t$ be the Hamiltonian paths generated by
	$F$, $\Th(F)$, and $\Th(F \# G)$, then
	\begin{align*}
		\overline{\Th(F\#G)}\,\#\,(\Th(F)\#\Th(G)) 
		&= (- \Th(F\#G) + \Th(F)\#\Th(G))\circ \wt{(fg)}_t\\
		&= \th \,\left(-G \circ {f_t}^{-1}  \circ \r \circ \pi + G \circ \r \circ \pi \circ \wt{f}_{t}^{-1}\right) 
		\circ \wt{(fg)}_t,
	\end{align*}
	which vanishes on $Z = Z \times 0$ where
	$$
	f_{t}^{-1}  \circ \r \circ \pi = \r \circ \pi \circ \wt{f}_{t}^{-1}
	$$
	by \eqref{e: rho and g}.  The proof for larger products is the same.	
\end{proof}

\begin{proof}[Proof of Lemma~\ref{l: PBZ}]
	Let $\th: \R^{k} \to [0,1]$ be any smooth bump function centered around $0$ with $\th(0) = 1$.
	By precomposing with $\Phi$, we can view $\th: W \to [0,1]$ as a bump function centered on $Z$.
	Since $\{H, \th H\} = 0$, by quasi-linearity of $\z$ we have 
	\begin{equation}\label{e: any th}
		\z(H) - \z(\th H) = \z((1-\th)H) = 0
	\end{equation}
	where the last equality follows since $Z$ is superheavy for $\z$ and
	$\th = 1$ on $Z$.
	
	Now let $\th_{\e}$ be a family of such bump functions where $\th_{\e}$ is
	supported on a ball of radius $\e$ in $\R^{k}$.  Since $H$ and $K$ Poisson commute with $\Phi$,
	\[
		\{\th_{\e}H, \th_{\e}K\} = \th_{\e}^{2} \{H, K\}
	\]
	and hence as $\e \to 0$ we have
	\[
		\lim_{\e \to 0} \norm{\{\th_{\e}H, \th_{\e}K\}} = \norm{\{H,K\}|_{Z}}.
	\]
	Therefore using the PB-inequality for $\z$ gives
	\begin{align*}
		\abs{\z(H+K) - \z(H) - \z(K)} &= \abs{\z(\th_{\e}H + \th_{\e}K) - \z(\th_{\e}H) - \z(\th_{\e}K)}\\
		&\leq C \sqrt{\norm{\{\th_{\e}H, \th_{\e}K\}}}.
	\end{align*}
	Taking the infimum over $\e$ of the upper bound gives \eqref{e: PBZ}.
	
	For the second claim, we use the same method.  Namely by \eqref{e: any th} and that
	$\z$ is Lipschitz in the $C^{0}$-norm,
	\[
		\abs{\z(H) - \z(K)} = \abs{\z(\th_{\e} H)- \z(\th_{\e} K)} \leq \norm{\th_{\e}H- \th_{\e}K}. 
	\]
	Since $H=K$ on $Z$, it follows that $\lim_{\e \to 0}\norm{\th_{\e}H- \th_{\e}K} = 0$
	and hence $\z(H) = \z(K)$.
\end{proof}

\begin{proof}[Proof of Lemma~\ref{l: l2}]
	By assumption we have a homotopy
	$\vp_{t}^{s}$ of loops based at $\ind$ in $Ham(W_{1})$,
	between 
	$\phi_{F} = \{\vp^0_t\}_{t\in [0,1]}$ and the constant loop 
	$\ind = \{\vp^1_t\}_{t\in [0,1]}$. 
	For $s$ fixed, let $F_t^s:  W_{1} \to \R$ be the 
	Hamiltonian in $\P\H(W_{1})$ generating
	the Hamiltonian loops $\{\vp^s_t\}_{t\in [0,1]}$ in $Ham(W_{1})$, via
	\[
		\d_t \vp_t^s = (\sgrad F_t^s)_{\vp_t^s} \quad
		\mbox{with initial condition}\quad \vp_0^s = \ind.
	\]
	Note that $F_t^0 = F_t$ and $F_t^1 = 0$.
	While for $t$ fixed, let $G_t^s:  W_{1} \to \R$ be the 
	Hamiltonian in $\P\H(W_{1})$ generating the homotopy, namely
	the Hamiltonian path $\{\vp_t^s\}_{s \in [0,1]}$ in $Ham(W_{1})$, via
	\[
		\d_s \vp_t^s = (\sgrad G_t^s)_{\vp_t^s} \quad\mbox{with initial condition}\quad
		\{\vp_t^0\}_{t} = \phi_{F}.
	\]
	Note that since $\vp_{0}^{s} = \vp_{1}^{s} = \ind$, that $G_{0}^{s} = G_{1}^{s} = 0$.
	The two Hamiltonians are related by 
	\cite[Proposition I.1.1]{Ban78}
	\begin{equation}\label{dhf1}
		\d_s F_t^s = \d_t G_t^s + \{F_t^s, G_t^s\}.
	\end{equation}
	
	Fixing $s$, the Hamiltonian $\Th(F^s_t): W_{2} \to \R$ in $\P\H(W_{2})$ will 
	generate a Hamiltonian path $\{\psi^s_t\}_{t\in [0,1]}$ in $Ham(W_{2})$, via
	\[
		\d_t \psi_t^s = (\sgrad \Th(F_t^s))_{\psi_t^s} \quad\mbox{with initial condition}
		\quad \psi_0^s = \ind.
	\]
	As $s$ varies, $\psi_{t}^{s}$ will be a homotopy of Hamiltonian paths in $Ham(W_{2})$, between 
	the paths $\phi_{\Th(F)} = \{\psi_t^0\}_{t\in [0,1]}$ and $\ind = \{\psi_t^1\}_{t\in [0,1]}$.
	However this will not be a homotopy of loops, since in particular 
	$\phi_{\Th(F)}$ may not be a loop.
	Letting $t$ be fixed, let $H_t^s:  W_{2} \to \R$ be the 
	Hamiltonian in $\P\H(W_{2})$ generating the Hamiltonian path $\{\psi_t^s\}_{s \in [0,1]}$ in 
	$Ham(W_{2})$, via
	\[
		\d_s \psi_t^s = (\sgrad H_t^s)_{\psi_t^s} \quad \mbox{with initial condition}
		\quad \{\psi_t^0\}_{t} = \phi_{\Th(F)}.
	\]
	
	\begin{figure}[h]
	\psfrag{I}{$\ind$}
	\psfrag{phi}{$\vp_{t}^{0}$}
	\psfrag{psi}{$\psi_{t}^{0}$}
	\psfrag{s}{$s$}
	\psfrag{t}{$t$}
	\psfrag{F}{$F_{t}^{s}$}
	\psfrag{G}{$G_{t}^{s}$}
	\psfrag{H}{$H_{t}^{s}$}
	\psfrag{F'}{$\Th(F_{t}^{s})$}
	\psfrag{11}{$\psi_{1}^{1}$}
	\psfrag{eta}{$\psi_{1}^{s}$}
	\psfrag{01}{$\psi_{0}^{1}$}
	\psfrag{00}{$\psi_{0}^{0}$}
	\psfrag{10}{$\psi_{1}^{0}$}
	\psfrag{p}{$\psi_{t}^{s}$}
	\psfrag{11'}{$\vp_{1}^{1}$}
	\psfrag{01'}{$\vp_{0}^{1}$}
	\psfrag{00'}{$\vp_{0}^{0}$}
	\psfrag{10'}{$\vp_{1}^{0}$}
	\psfrag{p'}{$\vp_{t}^{s}$}
	\psfrag{T}{$\Th$}

	\begin{center} 
	\leavevmode 
	\includegraphics[width=4in]{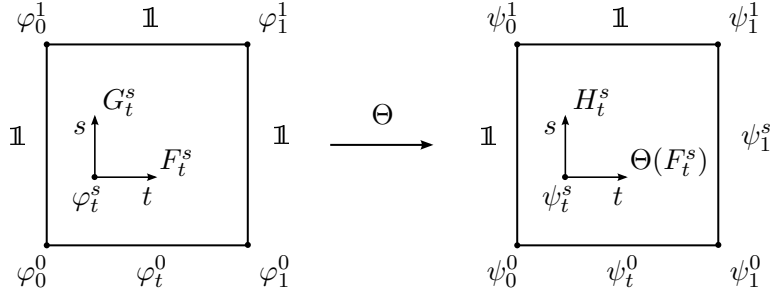}
	\end{center} 
	\caption{Left: The homotopy of loops in $Ham(W_{1})$ between $\phi_{F} = \{\vp_{t}^{0}\}_{t}$
	and the constant loop $\ind$.
	Right: The two parameter family in $Ham(W_{2})$ given by
	applying $\Th$ to the homotopy of loops.
	We construct a homotopy $\Psi$ with fixed endpoints 
	between the paths $\phi_{\Th(F)} = \{\psi_{t}^{0}\}_{t}$
	and $\eta = \{\psi_{1}^{1-u}\}_{u}$ in $Ham(W_{2})$.} 
	\end{figure}	
	
	Just as in \eqref{dhf1}, we have that $H_{\cdot}^{s} = L: W_{2} \times [0,1] \to \R$ is solution to the PDE
	\begin{equation}\label{dhf2}
		\d_s \Th(F_t^s) = \d_t L + \{\Th(F_t^s), L\}.
	\end{equation}
	By the assumption that $(\sgrad \Th(F_{t}^{s}))_{z} \in T_{z}Z$ for all $z \in Z$, it follows that
	\eqref{dhf2} can also be seen as a PDE for functions $L: Z \times [0,1] \to \R$.
	Applying $\Th$ to \eqref{dhf1} gives
	\begin{equation}\label{dhf3}
		\d_{s}\Th(F_{t}^{s}) = \d_{t}\Th(G_{t}^{s}) + \Th(\{F_{t}^{s}, G_{t}^{s}\})
	\end{equation}
	which on $Z$ becomes
	\begin{equation}\label{dhf4}
		\d_{s}\Th(F_{t}^{s}) = \d_{t}\Th(G_{t}^{s}) + \{\Th(F_{t}^{s}), \Th(G_{t}^{s})\}.
	\end{equation}
	Therefore the Hamiltonians $H_{\cdot}^{s}$ and $\Th(G_{\cdot}^{s})$
	both satisfy the PDE \eqref{dhf2} as functions $Z \times [0,1] \to \R$.
	Since we have the boundary data $H_{0}^{s} = \Th(G_{0}^{s}) = 0$ on $Z \times 0$,
	it follows by the method of characteristics for PDEs that 
	$$
		H_{\cdot}^{s} = \Th(G_{\cdot}^{s}) : Z \times [0,1] \to \R.
	$$ 
	Since $G_{1}^{s} = 0$ it follows that $H_{1}^{s}$ vanishes on $Z$.

	The path of normalized Hamiltonians $\{-H_{1}^{1-u}\}_{u \in [0,1]}$ generates the path
	$\{\eta_u = \psi_1^{1-u}\}_{u \in [0,1]}$ in $Ham(W_{2})$, which
	starts at $\eta_0= \psi_{1}^{0} = \ind$, using that $\psi_t^1 = \ind$, and ends at $\eta_1 = \psi_1^0$.
	Observe that $\Psi_t^s = \psi_t^s\eta_{st}$ is a homotopy
	of paths in $Ham(W_{2})$, between the paths
	\[
		\phi_{\Th(F)} = \{\Psi^0_t = \psi^0_t\}_{t \in [0,1]}
		\quad \mbox{and} \quad \eta = \{\Psi_t^1 = \eta_t\}_{t \in [0,1]}.
	\]
	The homotopy of paths $\Psi$ fixes the endpoints
	\[
		\Psi^s_0 = \ind \quad \mbox{and} \quad \Psi_1^s = \psi_1^s \eta_s = \psi_1^0,
	\]
	so we have proved that $[\phi_{\Th(F)}] = [\eta]$ in $\CH(W_{2})$.  Since $\eta$ is generated by 
	the path of normalized Hamiltonians $-H^{1-u}_{1}$, which vanish on $Z$, we are done.
\end{proof}

\subsection{Proof of Proposition~\ref{p: superheavy qm}}\label{s: shqm}

Recall that here $(M^{2n}, \w)$ is a closed symplectic manifold,
$\mu: \CH(M) \to \R$ is a stable homogenous quasi-morphism, and $X \subset M$
is a closed superheavy set for $\z_{\mu}$.

\begin{proof}[Proof of Proposition~\ref{p: superheavy qm}]
	We will first prove \eqref{e: shqm}.
	Let $F \in \P\H(M)$ be such that $F_{t}|_{X} \leq C$ for all $t$ and let $G \in C^{\infty}(M)$
	be such that $G|_{X} = C$ and $F_{t} \leq G$.  
	Since $X$ is superheavy for $\z_{\mu}$ it follows that
	\[
		C = \z_{\mu}(G) = \frac{\int_{M}G\,\w^{n} - \mu(\phi_{G_{n}})}{\vo(M)}
	\]
	so $\mu(\phi_{G_{n}}) = \int_{M}G\,\w^{n} - C\,\vo(M)$.  Therefore by stability
	\[		
	\mu(\phi_{G_{n}}) - \mu(\phi_{F}) 
	 \leq \vo(M) \int_{0}^{1}\max(F_{t}-G_{n}) dt \leq \int_{M}G\, \w^{n}
	\]
	and hence
	\[
		-C\vo(M) \leq \mu(\phi_{F}).
	\]
	The other inequality in \eqref{e: shqm} is proved similarly.
	
	Now let us prove that $\mu$ restricted to $\CH_{U}(M)$ is the Calabi homomorphisms,
	where $U = M \setminus X$.
	Given $F: M \times [0,1] \to \R$ where $\supp(F_{t}) \subset U$,
	pick an $H \in \H(M)$ such that $\supp(dH) \cap \supp(F_{t}) = \emptyset$ and $H|_{X} = 1$.
	It follows from \eqref{e: shqm} that $\mu(\phi_{H}) = -\vo(M)$, and if one sets
	$\l(t) = \tfrac{1}{\vo(M)}\,\int_{M} F_{t}\,\w^{n}$, then it follows from
	\cite[Lemma 22]{Bor10} that
	\[
		\mu(\phi_{\l H}) = -\vo(M)\,\int_{0}^{1}\l(t)\, dt = -\int_{0}^{1}\int_{M} F_{t}\,\w^{n} = -\Cal(F). 
	\]
	The normalized $(F_{t})_{n} = F_{t} - \l(t)$ is equal to $-\l(t)$ on $X$. 
	Since by design $\{F_{t}, H\} = 0$, it follows that $\l H \# (F)_{n} = \l H + (F)_{n}$ vanishes
	on $X$ and therefore $\mu(\phi_{\l H \# (F)_{n}}) = 0$.  
	Since $\phi_{\l H}$ and $\phi_{(F)_{n}}$ commute, using
	that quasi-morphisms are homomorphisms on commuting elements we have
	\[
		\mu(\phi_{(F)_{n}}) = \mu(\phi_{\l H \# (F)_{n}}) - \mu(\phi_{\l H}) = \Cal(F)
	\]
	so $\mu$ restricts to the Calabi homomorphism on $\CH_{U}(M)$.
\end{proof}

\subsection{The relation between Theorem~\ref{t: main} and the results in \cite{Bor10}}
Let us briefly explain the relation between Theorem~\ref{t: main} and the main results
in the paper \cite[Theorems 4 and 5]{Bor10}.  In the setting of \cite[Section 3.1]{Bor10}, 
one has a symplectic quasi-state or 	quasi-morphism on a symplectic disk bundle 
$(E, \w) \to (\Si, \s)$ that is build from a prequantization space for the closed symplectic manifold 
$(\Si, \s)$.  The disk bundle has a radial function $r^{2}: (E, \w) \to \R$, that induces a free $S^{1}$-action away from the zero section and performing symplectic reduction on a level set of 
$r^{2}$ recovers $(\Si, \s)$ up to scaling the symplectic form.  Therefore if one knew that a certain radial level of the disk bundle was superheavy, then one could apply Theorem~\ref{t: main} to achieve the results of  \cite[Theorems 4 and 5]{Bor10}.  
However in \cite{Bor10} such knowledge about
superheavy level sets is not required, instead due to the special relationship between 
$\w$ and $\s$, we are able to build a function $\Th: C^{\infty}(\Si) \to C^{\infty}(E)$ that globally preserves Poisson commutativity.  Furthermore, the failure of $\Th$ to be a Lie algebra homomorphism can be localized arbitrarily close to the boundary of the disk bundle and in this way one can ensure that any failure happens in a small open region whose complement is superheavy.

In contrast to the global $\Th$ in \cite{Bor10}, in this paper we work locally on the level set $Z$ and take full advantage of the fact that $Z$ is superheavy.  This is epitomized by the proof of 
Lemma~\ref{l: PBZ} where we only need that $\Th$ is a Lie algebra homomorphism on the level set $Z$.
In the setting of Theorem~\ref{t: main}, in general it is impossible to build a $\Th$ that preserves Poisson commutativity off of $Z$, due to the interaction between $\pi^{*}\w_{0}$ and the curvature terms
$d(r_{i}\,\pi^{*}\a_{i})$ from \eqref{e: local}, which perturb the symplectic form as one moves away from
$Z$.  This complication does not occur in the case studied in \cite{Bor10}, where what happens is equivalent to $d\a_{i}$ being a scalar multiple of $\w_{0}$.


\section{Proving Theorem~\ref{t: infinite qs} and Theorem~\ref{t: cpn blow up}}

As demonstrated by Abreu--Macarini \cite[Application 7]{AbrMac11}, one can prove that each fiber in the interval in Figure~\ref{f: 2blowup} is non-displaceable using that the fiber near a small blowup of $\CP^{2}$ is
non-displaceable.  As we will explain, their construction generalizes to the higher dimensional examples
$(Y^{2n}, \w_{\a}, \Delta_{\a}^n)$ that appear in Theorem~\ref{t: infinite qs}.
However, in order to invoke Theorem~\ref{t: main} to prove Theorem~\ref{t: infinite qs}, we will need
to prove that the fiber near a small blowup of $\CP^{n}$ is superheavy for a spectral quasi-state and quasi-morphism, which is a special case of the Theorem~\ref{t: cpn blow up} where $k=0$.
We will prove Theorem~\ref{t: cpn blow up} in Section~\ref{s: fsh} and we will then prove 
Theorem~\ref{t: infinite qs} in Section~\ref{s: AM imply T2}.

\subsection{Finding superheavy level sets}\label{s: fsh}

Ostrover and Tyomkin in \cite[Corollary F]{OstTyo09} showed that a small blowup of $\CP^{n}$ has two distinct spectral quasi-states and quasi-morphisms.  They prove it for $n=2$, but their method generalizes.
Their proof proceeds by computing the quasi-morphisms on a loop of Hamiltonian diffeomorphisms generated by the torus action using McDuff and Tolman's \cite{McDTol06} computation of the Seidel element.  In fact Ostrover and Tyomkin's method, when combined with McDuff's method of probes \cite[Lemma 2.4]{McD09}, can be effectively used to identify superheavy fibers of moment maps of symplectic toric Fano manifolds.

Ostrover and Tyomkin begin by finding a nice presentation of the quantum cohomology ring for
symplectic toric Fano manifolds, which allows one to read off the idempotents and the field summands.
Recall that a symplectic toric manifold $(M, \w)$ is Fano if it is deformation equivalent through toric structures to one that is monotone.
Denote by
$$\K^{\uparrow} = \left\{\sum_{\l\in\R} a_{\l}s^{\l} \mid a_{\l} \in \C, \mbox{and $\{\l \mid a_{\l} \not=0\}
\subset \R$ is discrete and bounded below}\right\}$$
the algebraically complete field of generalized Laurent series in the variable $s$.
This field has a non-Archimedian valuation
\begin{align}\label{e: valuation}
	\nu: \K^{\uparrow} \to \R \cup \{-\infty\} \quad&\mbox{where}\quad \nu\left(\sum a_{\l}s^{\l}\right)
	= - \min(\l \mid a_{\l} \not=0),
\end{align}
where under the convention that $\nu(0) = -\infty$, one has
$$ \nu(x+y) \leq \max(\nu(x), \nu(y)) \quad\mbox{and}\quad \nu(xy) = \nu(x)+\nu(y).$$
There is an isomorphic field $\K^{\downarrow}$,
where one replaces `bounded below' with `bounded above' and the valuation is defined in terms of 
max instead of $-\min$.  With $\K^{\uparrow}$ one can form the graded Novikov ring $\La^{\uparrow} = \K^{\uparrow}[q,q^{-1}]$ where $\deg(q) = 2$ and $\deg(s) = 0$.  
As a graded module over $\La^{\uparrow}$, the quantum cohomology of a symplectic manifold $(M^{2n}, \w)$ is given by
\[
	QH^{*}(M, \w) = H^{*}(M;\C) \otimes \La^{\uparrow}
\]
and its ring structure is a deformation of the normal cup product by Gromov-Witten invariants \cite{McDSal04}. The convention is to define spectral quasi-states and quasi-morphisms in terms of
quantum homology \cite{EntPol08, Ush11}, using idempotents in the $\K^{\downarrow}$-algebra $QH_{2n}(M, \w)$ that split off a field summand.  However by Poincar\'e duality,
one may just as well talk about idempotents in the $\K^{\uparrow}$-algebra $QH^{0}(M, \w)$ that give a field summand.  Since the results in \cite{OstTyo09} we need are stated in terms of $QH^{0}(M, \w)$,
we will adopt this perspective as well, so from now on $\K = \K^{\uparrow}$.

Consider a symplectic toric manifold $(M^{2n}, \w)$ with moment polytope
\[
	\Delta_{\w} = \{x\in \R^{n} \mid \ip{x, \xi^{j}} + a_{j} \geq 0\,,\,\mbox{for $j=1, \ldots, d$}\}
\]
where $a_{j} \in \R$ and $\xi^{j} = (\xi^{j}_{1}, \dots, \xi^{j}_{n}) \in \Z^{n}$ are the primitive interior conormal vectors for the
$d$ facets.  The \emph{Landau--Ginzburg superpotential} is given by
\[
	W_{\w}: (\K^{*})^{n} \to \K \quad\mbox{where}\quad 
	W_{\w}(y_{1}, \dots, y_{n}) = \sum_{j=1}^{d} s^{a_{j}} y_{1}^{\xi^{j}_{1}}\dots y_{n}^{\xi^{j}_{n}}.	
\]
The proof of the following theorem appears in Ostrover--Tyomkin \cite{OstTyo09}.
The first part summarizes \cite[Proposition 3.3, Corollary 3.6, Theorem 4.3]{OstTyo09}
and similar results appear in Fukaya--Oh--Ohta--Ono \cite[Theorem 6.1]{FukOhOht10a}.  
The proof of the isomorphism in \eqref{e: qh iso} proceeds by proving the right hand side is isomorphic to Batyrev's \cite{Bat93} combinatorial definition of quantum cohomology, which in the Fano case is isomorphic to quantum cohomology by Givental \cite{Giv95}.  The second part summarizes Ostrover--Tyomkin's
discussion in \cite[Section 6]{OstTyo09}, which consists of reinterpreting McDuff--Tolman's \cite{McDTol06} computation of the Seidel element \cite{Sei97} in terms of \eqref{e: qh iso}.  

\begin{theorem}[\cite{FukOhOht10a, McDTol06, OstTyo09}]\label{t: OT}
	If $(M^{2n}, \w)$ is Fano, then as $\K$-algebras
	\begin{equation}\label{e: qh iso}
		QH^{0}(M, \w) \cong \K[y_{1}^{\pm}, \ldots, y_{n}^{\pm}]/J_{W_{\w}}
	\end{equation}
	where $J_{W_{\w}}$ is the ideal generated by all partial derivatives of $W_{\w}$.
	Field direct summands in $QH^{0}(M, \w)$ correspond to non-degenerate critical points
	of $W_{\w}$ in $(\K^{*})^{n}$ and semi-simplicity of $QH^{0}(M, \w)$ is equivalent to
	all the critical points being non-degenerate.  If the potential for the monotone
	$W_{\w_{0}}$ has only non-degenerate critical points, then the same holds $W_{\w}$.
	
	Let $p = (p_{1}, \ldots, p_{n}) \in (\K^{*})^{n}$ be a non-degenerate critical point of $W_{\w}$,
	let $e_{p} \in QH^{0}(M, \w)$ be the corresponding idempotent, and let
	$\z_{e_{p}}$ be the associated spectral quasi-state.  Viewing the coordinate
	$x_{i}$ from the moment polytope as a Hamiltonian on $M$, which generates an $S^{1}$-action,
	we have that
	\begin{equation}\label{e: qs compute}
		\z_{e_{p}}(x_{i}) = -\nu(p_{i})
	\end{equation}
	where $\nu: \K \to \R \cup \{-\infty\}$ is the valuation on $\K$ from \eqref{e: valuation}.
\end{theorem}

As an example of using Ostrover--Tyomkin's method we will prove Theorem~\ref{t: cpn blow up},
and let us note the first part of our proof mimics their proof of \cite[Corollary F]{OstTyo09}.

\begin{figure}[h]
\psfrag{1}{$x_{2}$}
\psfrag{2}{$1$}
\psfrag{3}{$\tfrac{1}{3}$}
\psfrag{4}{$\lambda$}
\psfrag{5}{$0$}
\psfrag{6}{$x_{1}$}
	\begin{center} 
	\leavevmode 
	\includegraphics[width= 1.5in]{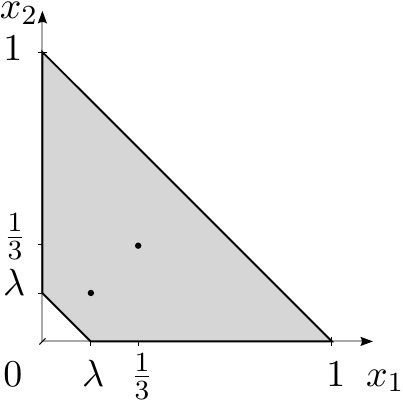}
	\end{center} 
\caption{The moment polytope $\Delta_{0,\l}^{2}$ for 
$(X^{4}_{0}, \w_{\l}) = (\CP^{2} \#\overline{\CP}^{2}, \w_{\l})$
and the two superheavy fibers.}
\end{figure}

\begin{proof}[Proof of Theorem~\ref{t: cpn blow up}]
	Recall that we are looking at $(X^{2n}_{k}, \w_{\l})$, 
	a small blowup of $\CP^{n}$ at a $k$-dimensional face, 
	with moment polytope
	\[
		\Delta_{k,\l}^{n} = 
		\left\{(x_{1}, \ldots, x_{n}) \in \R^{n} \mid x_{i} \geq 0\,,\,-\sum_{i=1}^{n} x_{i} + 1 \geq 0\,,\,
		\sum_{i=k+1}^{n} x_{i} - \l \geq 0\right\}.
	\]
	First consider the monotone case where $\l_{0} = \frac{n-k-1}{n+1}$. 
	All the critical points of $W_{\l_{0}}$ are non-degenerate
	and hence 
	$$W_{\l} = y_{1} + \cdots + y_{n} + s\, y_{1}^{-1}\cdots y_{n}^{-1} + s^{-\l}y_{k+1}\cdots y_{n}$$ 	
	in general will only have non-degenerate critical points.
	  
	Setting the partial derivatives $\d_{y_{i}}W_{\l} = 0$ equal to zero and scaling by $y_{i}$ gives
	\begin{align}
		&y_{i} - s\, y_{1}^{-1}\cdots y_{n}^{-1} = 0 &\mbox{for $i \leq k$} \label{e:<=k}\\
		&y_{i} - s\, y_{1}^{-1}\cdots y_{n}^{-1} + s^{-\l}y_{k+1}\cdots y_{n}= 0 &\mbox{for $i > k$}.
		\label{e:>k}
	\end{align}
	Therefore each critical point $p = (p_{1}, \dots, p_{n})$ lies on 
	$y = p_{1} = \cdots = p_{k}$ and $z= p_{k+1} = \cdots = p_{n}$
	for $y,z \in \K^{*}$, while equations \eqref{e:<=k} and \eqref{e:>k} become
	\begin{equation}\label{e: val}
		y^{k+1} z^{n-k}= s \quad\mbox{and hence }\quad \nu(y) = -\tfrac{1}{k+1}(1+(n-k)\,\nu(z))
	\end{equation}
	and
	\begin{equation}\label{e: poly}
		(z + s^{-\l}z^{n-k})^{k+1} - s z^{-(n-k)} = 0.
	\end{equation}
	The Newton diagram method \cite[Chapter 4, Section 3]{Wal78} can now be used to find the valuation 
	of the roots $z \in \K$ of the polynomial \eqref{e: poly}. 
	See Figures~\ref{f:Nb} and \ref{f:Nc} for example Newton diagrams,
	where a non-zero term $a z^{l}$ in the polynomial \eqref{e: poly} corresponds
	to a point at $(l, \nu(a))$ in the diagram and the slopes 
	represent the negative valuations $-\nu$ of the different roots.
	
	\begin{figure}[h]
	\psfrag{1}{$-n$}
	\psfrag{2}{$\lambda$}
	\psfrag{3}{$-1$}
	\psfrag{4}{$n$}
	\psfrag{5}{$1$}
	\begin{center} 
	\leavevmode 
	\includegraphics[width=6.5in]{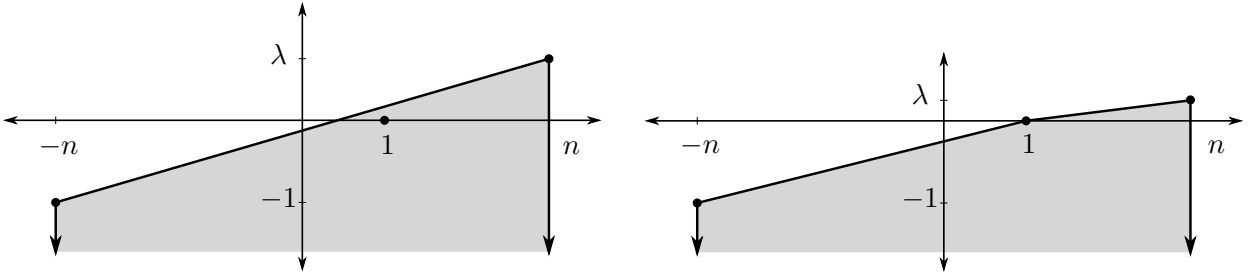}
	\end{center} 
	\caption{The Newton diagrams for the polynomial \eqref{e: poly} where $k=0$,
	drawn with $n=3$.
	Left: $\l \geq \tfrac{n-1}{n+1}$ and the slope is $\tfrac{1+\l}{2n}$.  
	Right: $\l < \tfrac{n-1}{n+1}$ and the slopes are $\tfrac{1}{n+1}$ and $\tfrac{\l}{n-1}$.}
	\label{f:Nb}
	\end{figure}
	
	\begin{figure}[h]
	\psfrag{1}{$-n+1$}
	\psfrag{2}{$2\lambda$}
	\psfrag{0}{$\lambda$}
	\psfrag{3}{$-1$}
	\psfrag{4}{$2n-2$}
	\psfrag{5}{$2$}
	\psfrag{6}{$n$}
	\begin{center} 
	\leavevmode 
	\includegraphics[width=6.5in]{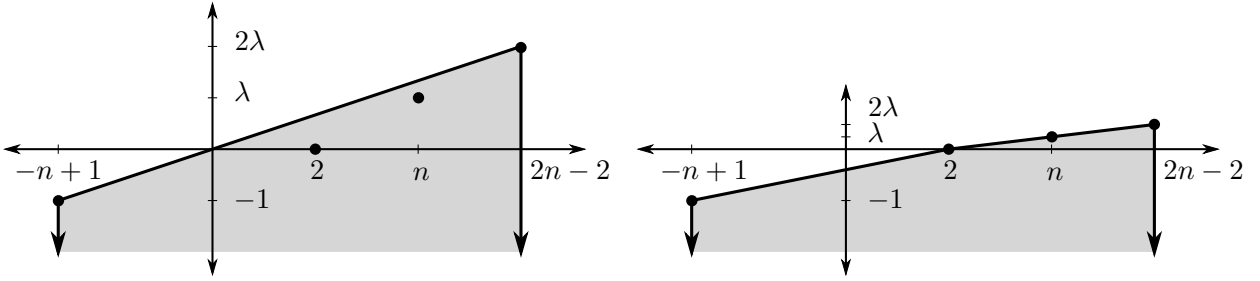}
	\end{center} 
	\caption{The Newton diagrams for the polynomial \eqref{e: poly} where $k=1$,
	drawn with $n=4$.
	Left: $\l \geq \tfrac{n-2}{n+1}$ and the slope is $\tfrac{1+2\l}{3(n-1)}$.  
	Right: $\l < \tfrac{n-2}{n+1}$ and the slopes are $\tfrac{1}{n+1}$ and $\tfrac{\l}{n-2}$.} 
	\label{f:Nc}
	\end{figure}

	If $\l \geq \tfrac{n-k-1}{n+1}$, then there are $(n-k)(k+2)$ roots $z$ of
	\eqref{e: poly} all with the valuation 
	$$-\nu(z) = \tfrac{1+\l(k+1)}{(n-k)(k+2)}.$$
	By \eqref{e: val} it follows $y$ has valuation $-\nu(y) = \tfrac{1-\l}{k+2}$,
	and the valuations of the pair $(y,z)$ determine the coordinates of the fiber $L_{s}$ in \eqref{e: SiB}.
	
	For positive $\l < \frac{n-k-1}{n+1}$, there are two types of roots:
	$n+1$ roots of type $z'$ and $(n-k-1)(k+1)$ roots of type $z''$, with valuations
	\begin{equation}\label{e:2roots}
		-\nu(z') = \tfrac{1}{n+1} \quad\mbox{and}\quad -\nu(z'') = \tfrac{\l}{n-k-1}.
	\end{equation}
	By \eqref{e: val}, the valuation for a $y'$ corresponding with $z'$ and 
	a $y''$ corresponding with $z''$ are
	\begin{equation}\label{e:2rootsy}
		-\nu(y') = \tfrac{1}{n+1} \quad\mbox{and}\quad -\nu(y'') = \tfrac{1}{k+1}(1-\l\,\tfrac{n-k}{n-k-1})
	\end{equation}
	So for positive $\l < \frac{n-k-1}{n+1}$, there are two types of critical points 
	$p = (p_{1}, \dots, p_{k}, p_{k+1}, \dots, p_{n}) = (y,z)$ where
	$$
	\mbox{$p' = (y',z')$ \quad and \quad 	$p'' = (y'',z'')$}
	$$
	where $(y',z')$ and $(y'',z'')$ have valuations given by 
	\eqref{e:2roots} and \eqref{e:2rootsy}.
	Let $e_{p'}$ and $e_{p''}$ be the corresponding idempotents, then by Theorem~\ref{t: OT}
	we have that
	\begin{equation}\label{e: two qs}
	\z_{e_{p'}}(x_{i}) = \frac{1}{n+1} \quad \mbox{and} \quad \z_{e_{p''}}(x_{i}) = 
	\begin{cases}
		-\nu(y'') = \tfrac{1}{k+1}(1-\l\,\tfrac{n-k}{n-k-1}) &\mbox{if $i \leq k$}\\
		-\nu(z'') = \tfrac{\l}{n-k-1} &\mbox{if $i > k$}
	\end{cases}
	\end{equation}
	and note that these are precisely the coordinates of the two Lagrangian fibers
	$$
		L_{c} = \left\{x = \diag(\tfrac{1}{n+1})\right\} \quad\mbox{and}\quad
		L_{k,\l} = \left\{x_{1} = \dots = x_{k} = \tfrac{1}{k+1}(1-\l\,\tfrac{n-k}{n-k-1})\,,\,\,
		x_{k+1} = \dots = x_{n} = \tfrac{\l}{n-k-1}\right\}.
	$$
	
	McDuff's method of probes \cite[Lemma 2.4]{McD09} displaces every fiber
	in $\Delta_{k,\l}^{n}$ except for these two fibers.  So by the
	vanishing property and quasi-linearity, under the moment map
	$\z_{e_{p'}}$ and $\z_{e_{p''}}$ push forward to integrals on $\Delta_{k,\l}^{n}$ 
	supported on these two points.
	Therefore by \eqref{e: two qs}, the push forward of the quasi-states $\z_{e_{p'}}$ 
	and $\z_{e_{p''}}$ are Dirac delta functions on $\Delta_{k,\l}$ for the points that 
	$L_{c}$ and $L_{k,\l}$ are over, respectively.
	One can now use the criterion of \cite[Proposition 4.1]{EntPol09RS}
	to prove that $L_{c}$ is superheavy for $\z_{e_{p'}}$ and $L_{k,\l}$ is superheavy
	for $\z_{e_{p''}}$.
\end{proof}

The following is a corollary of Theorem~\ref{t: OT}.
\begin{corollary}\label{c: product}
	Let $(M_{1}^{2n_{1}}, \w_{1})$ and $(M_{2}^{2n_{2}}, \w_{2})$ be symplectic toric Fano
	manifolds 	and let 
	$e_{i} \in QH^{0}(M_{i}, \w_{i})$ be idempotents corresponding to non-degenerate critical points
	$p_{i} \in (\K^{*})^{n_{i}}$ of $W_{\w_{i}}$.
	Let $X_{i} \subset M_{i}$ be superheavy for the spectral quasi-states $\z_{e_{i}}$.
	
	Then $(p_{1}, p_{2}) \in (\K^{*})^{n_{1}+n_{2}}$ is a non-degenerate critical point 
	for $W_{\w_{1}\oplus\w_{2}} = W_{\w_{1}} + W_{\w_{2}}$,
	the potential function for $(M_{1} \times M_{2}, \w_{1} \oplus \w_{2})$,
	and it corresponds to the idempotent 
	$e_{1} \otimes e_{2} \in QH^{0}(M_{1} \times M_{2}, \w_{1} \otimes \w_{2})$.
	The spectral quasi-state $\z_{e_{1}\otimes e_{2}}$ is a product symplectic quasi-state
	for $\z_{e_{1}}$ and $\z_{e_{2}}$ in the sense of \eqref{e: product qs} and hence
	$$ X_{1} \times X_{2} \subset M_{1} \times M_{2} $$
	is superheavy for $\z_{e_{1}\otimes e_{2}}$.
	The analogous results holds for spectral quasi-morphisms as well.
\end{corollary}
\begin{proof}
	That $(p_{1}, p_{2}) \in (\K^{*})^{n_{1}+n_{2}}$ is a non-degenerate critical point for
	the potential function $W_{\w_{1}\oplus\w_{2}}$ and that it corresponds
	to the idempotent $e_{1} \otimes e_{2}$ follows from the definitions and construction
	of the isomorphism \eqref{e: qh iso}.
	Since it corresponds to a non-degenerate critical point, $e_{1} \otimes e_{2}$ splits off a field,
	and hence defines a symplectic quasi-state $\z_{e_{1}\otimes e_{2}}$.  That this
	is a product symplectic quasi-state in the sense of \eqref{e: product qs} and that
	$X_{1} \times X_{2}$ is superheavy for $\z_{e_{1}\otimes e_{2}}$ follows from 
	\cite[Theorems 1.7 and 5.1]{EntPol09RS}.
\end{proof}

\subsection{Using the Abreu--Macarini construction to prove Theorem~\ref{t: infinite qs}}\label{s: AM imply T2}

For a positive $\a < \tfrac{1}{n+1}$, let $(Y^{2n}, \w_{\a})$ be as in the introduction,
with moment polytope 
\begin{equation}\label{e: mpY1}
	\Delta_{\a}^{n} = \left\{(x_{1}, \ldots, x_{n}) \in \R^{n} \mid x_{j} \geq 0\,,
	\,\, -\sum_{j=1}^{n} x_{j} + 1 \geq 0\,,\,\, 
	\sum_{j=1}^{n} x_{j} - (n-1)\a \geq 0\,,\,\, -\sum_{j=2}^{n}x_{j} + n\a \geq 0\right\}
\end{equation}
Generalizing the Abreu--Macarini construction \cite[Application 7]{AbrMac11} for the $n=2$ case, we will show that for positive $\l < \tfrac{1-(n+1)\a}{2}$ the fiber
$$
\mbox{$L_{\l}^{n}$ over the point $(x_{1}, x_{2}, \ldots, x_{n}) = (\a+\l, \a, \ldots, \a)$}
$$ 
in the moment polytope $\Delta_{\a}^n$ for $(Y^{2n}, \w_{\a})$ is non-displaceable.  In fact we will
show that $(Y^{2n}, \w_{\a})$ can be obtained by performing symplectic reduction on a 
codimension $n$ regular level set of 
\begin{equation}\label{e: product of spaces}
 (\CP^{n}\#\overline{\CP}^{n}\times \CP^{n-1}\times \CP^{1}, \w_{\a,\l} = \w_{1,\a,\l} \oplus \w_{2,\a,\l}  \oplus \w_{3,\a,\l}),
\end{equation}
coming from an $n$-dimensional subtorus of the product torus $\T^{n} \times \T^{n-1} \times \T^{1}$ that acts on \eqref{e: product of spaces}.
This level set will contain a Lagrangian torus
\begin{equation}\label{e: product of tori}
L^{\a,\l} = L_{1}^{\a,\l} \times L_{2}^{\a, \l} \times L_{3}^{\a,\l},
\end{equation}
which is superheavy for a spectral quasi-state and quasi-morphism and $L^{\a,\l}$ will map
to $L_{\l}^{n} \subset Y^{2n}$ under the reduction map.  By Theorem~\ref{t: main} this will suffice
to prove Theorem~\ref{t: infinite qs}.

\begin{proof}[Proof of Theorem~\ref{t: infinite qs}]
By scaling a small blowup $(X^{2n}_{0}, \w_{\eta}) = (\CP^{n}\#\overline{\CP}^{n}, \w_{\eta})$ from 
Theorem~\ref{t: cpn blow up},
we can create a $(X^{2n}_{0}, \w_{1,\a,\l})$ whose moment polytope is given by
\[
	\left\{x \in \R^{n} \mid x_{j} \geq 0\,,\,-\sum_{j=1}^{n}x_{j} + C \geq 0\,,
	\,\, \sum_{j=1}^{n} x_{j} - (n-1)(\a + \l) \geq 0\,\right\}
\]
where $C \gg 0$ is some large constant and the fiber $L_{0,\eta}$ over $\diag(\a+\l)$, which is near the exceptional divisor, is superheavy for a spectral quasi-state.  
By shifting the $x_{2}, \ldots, x_{n}$ coordinates
down by $\l$, the moment polytope for $(X^{2n}_{0}, \w_{1,\a,\l})$ becomes
\[
	\Delta_{1,\a,\l} = 
	\left\{(x_{1}, \dots, x_{n}) \in \R^{n} \mid x_{1} \geq 0\,,\, x_{j} +\l \geq 0\mbox{ for $j \geq 2$}
	\,,\,-\sum_{j=1}^{n}x_{j} + C \geq 0\,,\,\, \sum_{j=1}^{n} x_{j} - (n-1)\a \geq 0 \right\}
\]
where $C \gg 0$, and by Theorem~\ref{t: cpn blow up} 
the fiber $L_{1}^{\a,\l} = \{(x_{1}, \ldots, x_{n}) = (\a+\l, \a, \ldots,\a)\}$
is superheavy for a spectral quasi-state $\z_{1}$ whose idempotent comes from Theorem~\ref{t: OT}.

Let $(\CP^{n-1}, \w_{2,\a,\l})$ have moment polytope given by $(y_{2}, \ldots, y_{n}) \in \R^{n-1}$
\[
	\Delta_{2,\a,\l} = 
	\left\{(y_{2}, \dots, y_{n}) \in \R^{n-1} \mid y_{j} \geq 0\,,\, -\sum_{j=2}^{n} y_{j} + n\a \geq 0\right\}
\]
and let $(\CP^{1}, \w_{3,\a,\eta})$ have moment polytope given by $z_{1} \in \R$
\[
	\Delta_{3,\a,\l} = \left\{z_{1} \in \R \mid -1+2n\a + 2\l \leq z_{1} \leq 1\right\}.
\]
The fibers $L_{2}^{\a,\l} = \{y_{2} = \cdots = y_{n} = \a\}$ and
$L_{3}^{\a,\l} = \{z_{1} = n\a+\l\}$ are the Clifford tori, which are stems,
so by Entov--Polterovich \cite[Theorem 1.8]{EntPol09RS}
they are superheavy for any symplectic quasi-state.
Let $\z_{2}$ and $\z_{3}$ be spectral quasi-states, whose
idempotents comes from Theorem~\ref{t: OT}, such that $L_{i}^{\a,\l}$ is superheavy for $\z_{i}$.
It follows from Corollary~\ref{c: product} that the product \eqref{e: product of tori} of these 
Lagrangian tori 
$$
	L^{\a,\l} = L^{\a,\l}_{1} \times L^{\a,\l}_{2} \times L^{\a,\l}_{3}
	= \{x_{1} = \a+\l\,,\,\, x_{2}=\dots=x_{n} = \a\,,\,\, y_{2}=\dots=y_{n}=\a\,,\,\, z_{1} = n\a+\l\}
$$ 
is superheavy for a spectral quasi-state and quasi-morphism on
$$
	\Big(X^{2n}_{0} \times \CP^{n-1} \times \CP^{1}\,,\,\, \w_{1,\a,\l} \oplus \w_{2,\a,\l}\oplus \w_{3,\a,\l}\,,\,\,
	\Delta_{1,\a,\l} \times \Delta_{2,\a,\l} \times \Delta_{3,\a,\l}\Big)
$$
the product space \eqref{e: product of spaces}.

\begin{figure}[h]
\psfrag{1}{$x_{2}$}
\psfrag{2}{$2\alpha$}
\psfrag{3}{$\alpha$}
\psfrag{4}{$\tfrac{1-\alpha}{2}$}
\psfrag{5}{$1$}
\psfrag{6}{$x_{1}$}
\psfrag{7}{$(X^{4}, \w_{1,\alpha,\lambda})$}
\psfrag{8}{$(\CP^{1}, \w_{2,\alpha,\lambda})$}
\psfrag{9}{$(\CP^{1}, \w_{3,\alpha,\lambda})$}
\psfrag{a}{$L_{\lambda}$}

\begin{center} 
\leavevmode 
\includegraphics[width=3.5in]{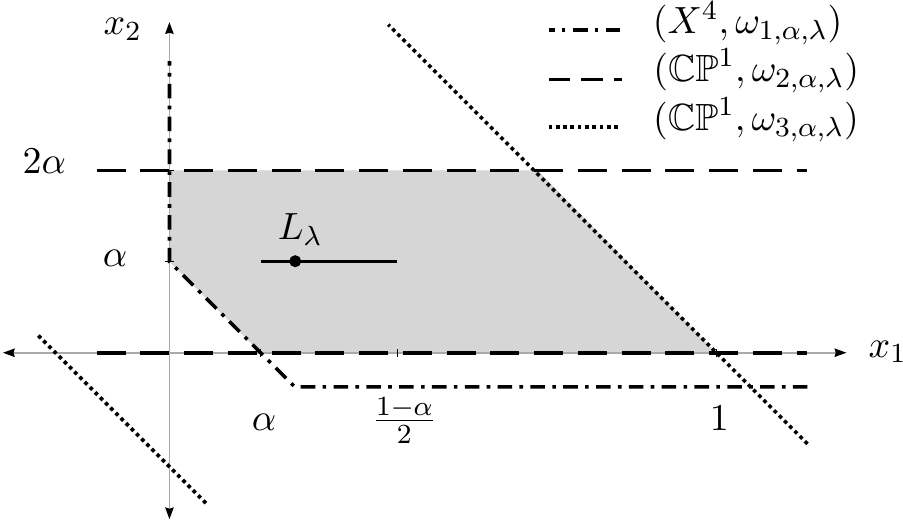}
\end{center} 
\caption{
The level set $Z$ projected
to the $(x_{1}, x_{2})$-coordinates.  This induces a Hamiltonian $\T^{2}$-action on the reduction of
$Z$, which is $(Y^{4},\w_{\a})$ and the fiber $L_{\l}^{n}$ is the reduction of $L^{\a,\l}\subset Z$.
Here $\a = \tfrac{1}{6}$ and $\l = \tfrac{1}{16}$.}
\label{left}
\end{figure}

\begin{figure}[h]
\psfrag{1}{$x_{2}$}
\psfrag{2}{$2\alpha$}
\psfrag{3}{$\alpha$}
\psfrag{4}{$\tfrac{1-\alpha}{2}$}
\psfrag{5}{$1$}
\psfrag{6}{$x_{1}$}
\psfrag{7}{$(X^{4}, \w_{1,\alpha,\lambda})$}
\psfrag{8}{$(\CP^{1}, \w_{2,\alpha,\lambda})$}
\psfrag{9}{$(\CP^{1}, \w_{3,\alpha,\lambda})$}
\psfrag{a}{$L_{\lambda}$}
\begin{center} 
\leavevmode 
\includegraphics[width=3.5in]{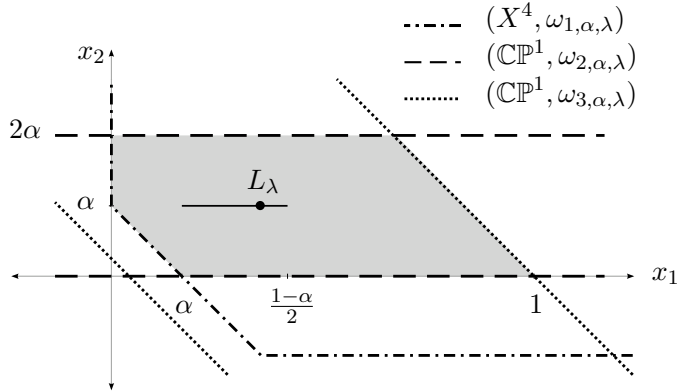}
\end{center} 
\caption{
The level set $Z$ projected
to the $(x_{1}, x_{2})$-coordinates.  This induces a Hamiltonian $\T^{2}$-action on the reduction of
$Z$, which is $(Y^{4},\w_{\a})$ and the fiber $L_{\l}^{n}$ is the reduction of $L^{\a,\l}\subset Z$.
Here $\a = \tfrac{1}{6}$ and $\l = \tfrac{3}{16}$.}
\label{right}
\end{figure}

The subset given by
$$
	Z = \left\{x_{2} = y_{2},\, \ldots,\, x_{n} = y_{n},\, \sum_{j=1}^{n}x_{n} = z_{1}\right\}
$$
is a regular level set of a Hamiltonian $\T^{n}$-action on the product space, and
$Z$ contains the superheavy Lagrangian torus $L^{\a,\l}$.  
The assumptions on $\l$ are used here to ensure that $Z$ is a regular level set.  
When $\l = 0$, the face given by
$x_{j} = -\l = 0$ aligns with the face $y_{j} = 0$, and when
$\l = \tfrac{1-(n+1)\a}{2}$, the face given by $z_{1} = -1+2n\a + 2\l$ aligns with
the face given by $\sum_{j=1}^{n} x_{j} = (n-1)\a$.  These alignments cause $Z$ to not be a regular
level set.  This behavior can be seen in Figures~\ref{left}
and \ref{right}, as $\l$ goes from small to large.

It follows from Theorem~\ref{t: main} that the reduction $(Z/\T^{n}, \bar{\w}_{\a,\l})$
inherits a symplectic quasi-state, $\z_{\l}$, and quasi-morphism, $\mu_{\l}$, 
and the reduction of $L^{\a,\l}$ is
superheavy.  The subtorus given by the action of the $x_{i}$'s is integrally transverse to
the subtorus giving the level set $Z$, and hence the moment polytope of $(Z/\T^{n}, \bar{\w}_{\a,\l})$
is given by the projection of $Z$ to the $(x_{1}, \ldots, x_{n})$-coordinates.  This is precisely the
moment polytope $\Delta_{\a}^n$ for $(Y^{2n}, \w_{\a})$ and the projection of $L^{\a,\l}$ gives
the fiber over $$(x_{1}, x_{2}, \ldots, x_{n}) = (\a+\l, \a, \ldots, \a)$$ which is the description of
$L_{\l}^{n} \subset (Y^{2n}, \w_{\a})$.  See Figures~\ref{left} and \ref{right} for examples.
Therefore it follows from Delzant classification of toric manifolds
\cite{Del88}, that $(Y^{2n}, \w_{\a}; L_{\l}^{n})$ is identified with $(Z/\T^{n}, \bar{\w}_{\a,\l}; L^{\a,\l}/\T^{n})$,
and hence $L_{\l}^{n}$ is superheavy for a symplectic quasi-state $\z_{\l}$ and quasi-morphism $\mu_{\l}$
on $(Y, \w_{\a})$.
\end{proof}


\bibliographystyle{alpha}
\bibliography{symplectic}

\end{document}